\documentclass[11pt,letterpaper]{amsart}
\usepackage{amsmath,amssymb,amsthm, epsfig, amsmath}
\usepackage{amsfonts}
\usepackage{graphicx}
\usepackage{hyperref}
\usepackage{xcolor}
\usepackage{tikz} 
\usepackage{mathtools}
\usepackage{changepage}

\usepackage{cleveref}

\hypersetup{
	colorlinks,
	urlcolor=azul,
	linkcolor=azul,
	citecolor=azul}

\newtheorem{theorem}{Theorem}

\definecolor{azul}{RGB}{33,64,154}

\newtheorem{corollary}[theorem]{Corollary}

\newtheorem{definition}[theorem]{Definition}

\newtheorem{lemma}[theorem]{Lemma}

\newtheorem{proposition}[theorem]{Proposition}
\newtheorem{remark}[theorem]{Remark}

\numberwithin{equation}{section}
\numberwithin{theorem}{section}

\newcommand{\supp}{\operatorname{supp}}

\newcommand{\vol}{\mathrm{vol}}

\title[Conformal geometry and Spectral Bounds]{Conformal geometry and Spectral Bounds on Manifolds with Boundary}

\author{Tiarlos Cruz  \and Leandro F. Pessoa \and Erisvaldo V\'eras}

\address{ Instituto de Matem\'atica, Universidade Federal de Alagoas, 57072-970, Macei\'o - Alagoas, Brazil}
 \email{cicero.cruz@im.ufal.br}

 \address{Departamento de Matemática, Universidade Federal do Piauí, 64049-550, Teresina - Piauí, Brazil}
 \email{leandropessoa@ufpi.edu.br}

 \address{Departamento de Matemática, Universidade Federal do Piauí,
 64049-550, Teresina - Piauí, Brazil}
 \email{erisvaldoveras35@gmail.com}

\thanks{T. Cruz was partially supported by CNPq, Brazil, under grant numbers 307419/2022-3, 408834/2023-4, 444531/2024-6,  403770/2024-6,  400078/2025-2 and CAPES/MATH-AMSUD 88887.985521/2024-00. L.F. Pessoa was partially supported by CNPq, Brazil, under grant numbers 306543/2022-2, 402563/2023-9 and 402713/2024-9. E. Véras was supported by Funda\c c\~ao de Amparo \`a Pesquisa do Piau\'i - FAPEPI/Brazil, Scholarship: 029/2021.
}
\subjclass[2020]{58J50, 53A30, 35P15}
\keywords{Steklov spectrum, Steklov-type problem, relative conformal volume, conformal Dirichlet-to-Robin map, Morse index}

\begin{document}

\begin{abstract}
This work investigates upper bounds for the spectrum of the Steklov-type operator on Riemannian manifolds with boundary. We extend the Fraser–Schoen estimate for the first positive Steklov eigenvalue to higher Steklov eigenvalues, in terms of the relative conformal volume and the isoperimetric ratio. Our approach, which draw on Korevaar's method, further developed by Grigor'yan-Netrusov-Yau and Kokarev, can be adapt to derive a Korevaar-type estimate for the conformal Dirichlet-to-Robin operator on the Euclidean ball, showing that its $k$-th eigenvalue is bounded from above by a multiple of $k^{2/n}$, as well as a corresponding bound in terms of the relative conformal volume for proper conformal immersion into the Euclidean ball. We also establish a lower bound for the number of negative eigenvalues of the Steklov-type problem in terms of the relative conformal volume, with applications to the spectrum of the conformal Dirichlet-to-Robin operator and to the Morse index of type-II stationary capillary hypersurfaces.
\end{abstract}

\maketitle
\tableofcontents

\section{Introduction}

The spectral theory of geometric operators on Riemannian manifolds lies at the interface of analysis, geometry, and topology. A fundamental example is the Laplace–Beltrami operator on a closed Riemannian manifold, whose first eigenvalue admits classical estimates derived from coordinate test functions in the Rayleigh quotient, following a normalization that positions the barycenter at the origin. This technique, known as Hersch’s trick, was introduced in \cite{He} and has been widely applied (see e.g. \cite{SI, LY, YY}).
While Hersch's method addresses only the first eigenvalue, the question of controlling higher eigenvalues required new ideas. 
A breakthrough was achieved by Korevaar \cite{K}, who established upper bounds for the $k$-th Laplace eigenvalue via a new approach,  later refined by Grigor'yan and Yau \cite{gy} and Grigor'yan, Netrusov, and Yau \cite{gny}.

For manifolds with boundary, one naturally considers spectral problems that reflect the geometry of the boundary. One of the most fundamental is the Steklov eigenvalue problem, first introduced by Vladimir Steklov in 1902 (cf. \cite{Stekloff}). The Steklov eigenvalues are highly sensitive to the boundary geometry and have recently gained prominence through their applications to free boundary minimal surfaces \cite{fs,fs2,etal,KS}.  
For compact surfaces, Fraser and Schoen \cite{fs} showed that the first Steklov eigenvalue is controlled by the topology of the surface and the length of its boundary, generalizing an earlier result of Weinstock in \cite{W}. This establishes a direct analogy with the well-known estimates for the first Laplace eigenvalue on closed surfaces \cite{H,YY}, revealing a deep parallel between the two spectral theories.
For further reading, we refer interested readers to the surveys \cite{CGGS, GPs, MLi}. 

A key contribution of Fraser and Schoen \cite{fs} was the introduction of the relative conformal volume, inspired by the conformal volume defined by Li and Yau \cite{LY}, and also referred to as visual volume by Gromov \cite{G}.
Assume that $(M,g)$ admits proper conformal immersions into the Euclidean unit ball $\mathbb{B}^m$. The \textit{relative $m$-conformal volume of $M$} is the  min-max invariant defined  by
\begin{equation*}
     V_{rc} (M,m)=\inf_{\psi}\sup_{f \in G} {\vol}_g( f (\psi (M))), 
\end{equation*}   
where the infimum is taken over all nondegenerate conformal maps $\psi \colon M \rightarrow \mathbb{B}^m$ satisfying $\psi(\partial M) \subset \partial \mathbb{B}^m$, and $G$ denotes the group of conformal diffeomorphisms of $\mathbb{B}^m$. Here $\vol_g$ refers to the volume with respect to the induced metric $g=f^*\psi^*g_0$, where $g_0$ is the canonical metric on $\mathbb{B}^m$. Building on this concept, the authors established in \cite[Theorem 6.2]{fs} a boundary analogue of the classical estimate proved by Li and Yau \cite{LY}, as well as El Soufi and Ilias \cite{SI}, for the first eigenvalue of the Laplace-Beltrami operator on closed manifolds. More precisely, they showed that\footnote{In contrast to the notation used in \cite{fs}, we adopt the convention that $\sigma_1 = 0$ is the first Steklov eigenvalue.}
\begin{equation}\label{fraser_schoen}
\sigma_2(M,g) \mbox{vol}_g(\partial M)\leqslant nV_{rc}(M,m)^{2/n}\mbox{vol}_g(M)^{\frac{n-2}{n}}.
\end{equation}

In this paper, we employ a unified approach to obtain upper bounds for higher eigenvalues of various elliptic boundary value problems on Riemannian manifolds with boundary. Our framework centers on the Steklov-type problem
\begin{equation*}
\begin{cases}
\begin{array}{rl}
-\Delta u + \mathfrak pu = 0 & \text{in} \ M,  \\
\frac{\partial u}{\partial \eta}-\mathfrak qu = \sigma  u  &   \text{on} \ \partial M,
\end{array}
\end{cases}
\end{equation*}
which incorporates potential functions $\mathfrak p$ and $\mathfrak q$  in both the interior and on the boundary,  providing a versatile setting for spectral analysis.

Our first main result provides an upper bound for higher Steklov eigenvalues (case $\mathfrak{p,q}\equiv0$) in terms of the relative conformal volume and the isoperimetric ratio, generalizing the first Steklov eigenvalue upper bound in \cite[Theorem 6.2]{fs} and providing a boundary analogue of Kokarev's result \cite{kok}.  More precisely, for every $k\geqslant 1$, we show that the $k$-th normalized Steklov eigenvalue of $M$ satisfies
\begin{equation}\label{esti_1_intro}
\bar \sigma_k(M,g) \leq C \frac{V_{rc}(M,m)^{\frac{2}{n}}}{I(M)^{\frac{n-2}{n-1}}}k^{2/n},
\end{equation}
where $C=C(n,m)>0$ and $I(M)=\vol_g(\partial M)/\vol_g(M)^{\frac{n-1}{n}}$ is the classical isoperimetric ratio (see Theorem \ref{thm:letterA}). Estimate \eqref{esti_1_intro} can also be regarded as a variant of Colbois, El Soufi, and Girouard \cite[Theorem 1]{CEG} expressed in terms of the relative conformal volume. For further results on higher Steklov eigenvalues in different geometric contexts, we refer to \cite{GPs, H,Ka}.

The proof of Theorem~\ref{thm:letterA}, and of the main results in this paper, follows the strategy originating in Korevaar’s work \cite{K}, later developed in \cite{gy,gny}, and relies crucially on the construction of test functions supported on disjoint annuli, allowing effective control of higher eigenvalues, see \cite{CEG,kok} and references therein for some applications. Building on ideas from Kokarev \cite{kok} (cf. \cite{SX}), we adapt these techniques using specific conformal diffeomorphisms of the Euclidean ball. A key difficulty in this geometric construction is to maintain control over the Rayleigh quotient while distributing test functions across appropriately chosen annular regions. The main novelty in our strategy is to consider minimal Lipschitz extensions (MLE), as they attain the smallest possible Lipschitz constant among all Lipschitz extensions (cf. \cite{McShane,Wh}).

We then study the spectrum of the conformal Dirichlet-to-Robin operator associated to the Escobar–Yamabe problem on manifolds with boundary \cite{E,E2} (case $\mathfrak{p}= c_nR_{\bar g}$, $\mathfrak{q}= b_nH_{\bar g}$), a setting in which spectral geometry and conformal invariants interact  profoundly. This problem concerns finding conformal metrics with prescribed scalar curvature and mean curvature of the boundary, generalizing the classical Yamabe problem. 
The conformal Dirichlet-to-Robin operator $B_{\bar g}$ exhibits rich spectral behavior, as recently investigated in \cite{Cox}. Motivated by these advances, we prove a Korevaar-type upper bound (Theorem \ref{thm:letterB}) for the spectrum of $B_{\bar g}$ on the Euclidean unit ball $(\mathbb{B}^n,g)$. Specifically, we show that the normalized $k$-th eigenvalue satisfies
\begin{eqnarray*}
\bar\sigma_k(\mathbb{B}^{n},B_{\bar g})\leqslant C(n)k^{\frac{2}{n}} \qquad \text{for all} \ \bar g \in [g],
\end{eqnarray*}
for some positive constant $C(n)$. Furthermore, we generalize the analysis to a broader class of manifolds using the notion of relative conformal volume (Theorem \ref{thm:letterC}), and prove a Hersch-type result (Theorem \ref{thm:letterD}) characterizing the maximizer of the first normalized eigenvalue of $B_{\bar g}$ on the ball, showing that the Euclidean metric is the unique maximizer up to scaling. These results are inspired, and extends the work of Sire and Xu \cite{SX} to manifolds with boundary.

A closely related object of study that has been extensively investigated in the literature \cite{gn,gns,kok,LY2,Lie} is the counting of negative eigenvalues of a Schr\"odinger operator defined on a closed Riemannian manifold. This problem consists in studying the following eigenvalue problem on $M$
$$
-\Delta u - \mathfrak p u = \lambda u,
$$
where $\mathfrak p \in L^\infty(M)$ is a given potential. The number of negative eigenvalues  $\operatorname{Neg}(\mathfrak p)$ for this operator was investigated by Grigor'yan, Nadirashvili, and Sire \cite{gns}, who established the following lower bound
\begin{equation*}
\operatorname{Neg}(\mathfrak p) \ge
\frac{C}{\mu(M)^{\frac{n}{2}-1}}
\left( \int_M \mathfrak p \, d\mu \right)_+^{\frac{n}{2}},
\end{equation*}
where $C>0$ is a constant that, in the case $n=2$, depends only on the genus of \(M\), and for $n>2$ depends only on the conformal class of $M$.

Our last main result extends the preceding eigenvalue counting estimates to a Steklov-type problem, where we derive an upper bound for the number of negative eigenvalues expressed in terms of the relative conformal volume (Theorem~\ref{thm:letterE}). This analysis yields two geometric applications. First, based on the fact that the signs of all its eigenvalues are conformally invariant (Proposition \ref{invariant}), we obtain a lower bound for the number of negative eigenvalues of the conformal Dirichlet-to-Robin map (Theorem \ref{conformal_neg}). Second, we estimate the Morse index\footnote{The Morse index measures the dimension of wetting-area-preserving deformations that decrease the area of a type-II stationary hypersurface.} of stationary hypersurfaces (Theorem \ref{morse_index}) in the type-II partitioning problem introduced by Bokowsky and Sperner \cite{BS}, and Burago and Maz'ya \cite{BM}, which seeks area-minimizing hypersurfaces that partition a domain $\Omega$ into two subdomains with prescribed wetting area (see also \cite{GX,TZ1,TZ2}).

The paper is structured as follows: in Section \ref{sec_tec}, we introduce the Steklov-type eigenvalue problem and its variational formulation. We also discuss several geometric contexts where this problem arises, and the spectral properties of the associated Dirichlet-to-Robin operator. Section \ref{sec_test-functions} is devoted to the construction of the test functions for the Rayleigh quotient which is based on the existence of a large collection of disjoint annuli that carry a controlled mass. In Section~\ref{sec_stek}, we establish upper bounds for higher eigenvalues. After proving a Korevaar-type estimate for normalized Steklov eigenvalues in terms of the relative conformal volume and the isoperimetric ratio (Theorem~\ref{thm:letterA}), we derive analogous bounds for the conformal Dirichlet-to-Robin operator, first on the unit Euclidean ball (Theorem~\ref{thm:letterB}) and then on manifolds admitting proper conformal immersions into $\mathbb{B}^m$ (Theorem~\ref{thm:letterC}). We conclude the section with a Hersch-type result for the first eigenvalue (Theorem~\ref{thm:letterD}). 

Section \ref{negative eigen} focuses on the study of negative eigenvalues of the Steklov-type problem. We derive a lower bound for the number of negative eigenvalues (Theorem \ref{thm:letterE}), and then apply this result to two geometric settings: to the conformal Dirichlet-to-Robin operator (Theorem \ref{conformal_neg}) and to stationary capillary minimal hypersurfaces (Theorem \ref{morse_index}). Finally, in the Appendices, we provide a detailed proof of the conformal invariance of the signs of the eigenvalues of the conformal Dirichlet-to-Robin operator (Proposition \ref{sign}), an upper bound for the relative conformal volume for surfaces that admit a free boundary minimal immersion into the hemisphere $S_+^m$ (Proposition \ref{prop_est_vrc_to_vol}), and a lower bound for the Morse index of stationary capillary minimal hypersurfaces through the spectrum of their Jacobi operator with Robin boundary condition (Theorem \ref{morse_index_thm}).

\medskip
\noindent\textbf{Acknowledgments:} T. Cruz and L. Pessoa thank Lucas Ambrozio and Instituto de Matem\'atica Pura e Aplicada (IMPA) for their hospitality during  the Summer Program 2025, where part of this work was carried out. They are also grateful to Vanderson Lima for useful discussions.

\section{The Steklov-type spectrum}
\label{sec_tec}

Consider a smooth Riemannian manifold $(M,g)$  with compact boundary $\partial M$. In this section, we present variational properties of the eigenvalues of the following Steklov-type eigenvalue problem:
\begin{equation}\label{eq_auxi}
\begin{cases}
\begin{array}{rl}
-\Delta u + \mathfrak pu = 0 & \text{in} \ M,  \\
\frac{\partial u}{\partial \eta}-\mathfrak qu = \sigma  u  &   \text{on} \ \partial M,
\end{array}
\end{cases}
\end{equation}
where  $\eta$ denotes the outward unit normal vector field along $\partial M$, and $\mathfrak p$, $\mathfrak q$ are given potential functions whose regularity will be specified later. The boundary value problem \eqref{eq_auxi} arises in various contexts, including: 
\medskip

\noindent\textbf{The Steklov problem}. 
The Steklov problem on $(M,g)$ is given by 
    \begin{equation}
\begin{cases}\label{stek_sec}
\begin{array}{rl}
\Delta_g u = 0 & \text{in} \ M,  \\
\frac{\partial u}{\partial \eta_g} = \sigma  u  &   \text{on} \ \partial M.
\end{array}
\end{cases}
\end{equation}
The Steklov eigenvalues arise as the spectrum of the Dirichlet-to-Neumann map  $D \colon H^{\frac{1}{2}}(\partial M) \to H^{-\frac{1}{2}}(\partial M)$, sometimes known as the voltage-to-current operator, which maps $u \mapsto \partial \bar{u}/\partial \eta,$ where $\bar{u}$ is the harmonic extension
of $u$ to $M$, see \cite{CGGS, GPs}. This map is a first-order elliptic pseudodifferential operator that is self-adjoint. Its spectrum is discrete, nonnegative, and unbounded from above. The boundary problem \eqref{stek_sec} has deep connections and strong consequences for the study of free boundary minimal hypersurfaces, see \cite{fs} for example.\medskip

\noindent\textbf{Escobar-Yamabe problem}. The Yamabe problem for a closed Riemannian manifold consists finding a metric of constant scalar curvature within a given conformal class. This generalizes the uniformization theorem for Riemann surfaces and was solved through the combined work of Trudinger \cite{T}, Aubin \cite{A}, and Schoen \cite{S}. If the compact manifold being considered has a nonempty boundary, Escobar proposed a similar problem in \cite{E1}. To fix notation, we denote by $R_g$ the scalar curvature of $(M,g)$ and by $H_g$ the mean curvature of the boundary $\partial M$. There are two natural ways to extend the Yamabe problem to manifolds with boundary:
\begin{itemize}
    \item[(I)] find a metric $\bar{g}$ in the conformal class of $g$ such that $R_{\bar{g}}$ is constant and $H_{\bar{g}}=0$;
    \item[(II)] find a metric $\bar{g}$ in the conformal class of $g$ such that $R_{\bar{g}}=0$ and $H_{\bar{g}}$ is constant.
\end{itemize}
Problem (II) was solved in a series of works by Escobar \cite{E, E1}, Marques \cite{Ma}, Chen \cite{CH}, and Mayer and Ndiaye \cite{MN}.

From the PDE perspective, problem (II) is closely related to the following eigenvalue problem associated with the conformal Dirichlet-to-Robin map:
\begin{equation}\label{esc}
\begin{cases}
\begin{array}{rl}
-\Delta_g u + \frac{n-2}{4(n-1)} R_gu = 0 & \text{ in }\  M, \\[0.2cm]
\displaystyle\frac{\partial u}{\partial\eta_g} + \frac{n-2}{2} H_gu=\sigma u & \text{ on } \ \partial M.
\end{array}
\end{cases} 
\end{equation} 
This boundary value problem is a special case of the broader class of conformally covariant maps, and it has been extensively studied in recent years (see, e.g. \cite{Cox, Ho, HP,SX} and the references therein).
\medskip

\noindent\textbf{The Jacobi-Steklov problem}. 
Let $\Omega$ be a convex body, that is, a compact convex set with non-empty interior. Two types of partitioning problems for convex bodies have been extensively studied  by Bokowsky and Sperner \cite{BS} and, in the special case $\Omega=\mathbb B^{n+1}\subset \mathbb{R}^{n+1}$, by Burago and Maz'ya \cite{BM}.
\medskip

\noindent\textit{Type-I partitioning problem:} find an area-minimizing hypersurface among all hypersurfaces in $\Omega$ that divide $\Omega$ into two disjoint subdomains $\Omega_1$ and $\Omega_2$ with  prescribed volumes
$$
\left|\Omega_1\right|=s|\Omega| \text { and }\left|\Omega_2\right|=(1-s)|\Omega|, \quad\quad s \in(0,1).
$$

\noindent\textit{Type-II partitioning problem:} find an area-minimizing hypersurface among all hypersurfaces in $\Omega$ that divide $\Omega$ into two disjoint subdomains, $\Omega_1$ and $\Omega_2$, with prescribed wetting areas, namely,
$$
\left|\Omega_1\cap \partial \Omega\right|=s|\Omega| \text { and }\left|\Omega_2\cap \partial \Omega\right|=(1-s)|\Omega|, \text { for some } s \in(0,1).
$$

In this work we focus on the Type-II partitioning problem. Geometrically, stationary hypersurfaces for the type-II partitioning problem in $\Omega$ are capillary minimal hypersurfaces, i.e., minimal hypersurfaces intersecting $\partial \Omega$ at a constant contact angle $\theta\in (0,\pi)$. In \cite{BS} and \cite{BM}, classification results and estimates for the associated isoperimetric ratio were obtained for general convex bodies.
Associated with the Type-II partitioning problem is the following Jacobi–Steklov eigenvalue problem (see also \cite{GX,Zhu}, and \cite{ACS2,BPM} for the case $\theta=\pi/2$)
\begin{equation*}
\begin{cases}
\begin{array}{rl}
\Delta_{M
} u + (\text{Ric}^{\Omega}(\nu,\nu)+|A^{M}|^2)u = 0 & \text{in} \ M,  \\
\displaystyle\frac{\partial u}{\partial \eta_g}-\left(\csc\theta A^{\partial \Omega} (\bar{\nu}, \bar{\nu}) + \cot \theta \, A^M(\eta, \eta)\right)u = \sigma u  &   \text{on} \ \partial M,
\end{array}
\end{cases}
\end{equation*}
 where $\text{Ric}^{\Omega}(\nu,\nu)$ is the Ricci curvature of $\Omega$ in the direction of the unit normal $\nu$ along $M$, $A^{M}$ and $A^{\partial \Omega}$ are, respectively, the second form fundamental of $M$ and $\partial \Omega$, $\eta$ is the unit outward normal to $\partial M$ in $M$, and  $\bar{\nu}$ is the unit normal to $\partial M$ in $\partial \Omega$. \medskip

With these motivations in place, we now turn to the variational characterization of the eigenvalues associated with the Steklov-type problem \eqref{eq_auxi} defined on an $n$-dimensional compact Riemannian manifold $(M,g)$ with non-empty boundary $\partial M$. Throughout the paper, we assume that the potential functions $\mathfrak p$ and $\mathfrak{q}$ satisfy the following regularity conditions:
\begin{equation}\label{regul_p_q_potent}
\mathfrak{p} \in L^p(M), \ \text{for } \ p>\frac{n}{2} \quad \text{and} \quad \mathfrak{q} \in L^\infty(\partial M).
\end{equation}

Let $\mathfrak L_g \doteq -\Delta_g + \mathfrak p$, and assume that 
$0$ is not a Dirichlet eigenvalue of $\mathfrak L_g$. 
Under this assumption, for any $u\in L^2(\partial M)$ for which there exists 
$\hat u\in H^1(M)$ satisfying 
\[
\operatorname{Tr}\hat u = u
\quad\text{and}\quad
\mathfrak L_g \hat u = 0 \ \text{in the weak sense},
\]
the function $\hat u$ is unique (e.g. \cite{AM,TRAN}).  
Here $\operatorname{Tr}:H^1(M,g)\to L^2(\partial M,g)$ denotes the trace operator, 
whose range is the standard trace space $H^{1/2}(\partial M)\subset L^2(\partial M)$. We recall that, for $u \in L^{2}(\partial M)$, a function $\hat u \in H^{1}(M)$ with $\operatorname{Tr}\hat u = u$ is said to satisfy $\mathfrak{L}_g \hat u = 0$ in the weak sense if
\begin{equation}\label{weak_form_Lg}
\int_M \langle\nabla \hat u , \nabla \varphi\rangle_g \, dv_g
+ \int_M \mathfrak p\, \hat u\, \varphi \, dv_g = 0,
\qquad
\forall\, \varphi \in H^1_0(M).
\end{equation}

We are thus naturally led to define the boundary operator associated to 
\eqref{eq_auxi} only on those boundary functions $u\in L^2(\partial M)$ that admit 
a weak solution $\hat u\in H^1(M)$ of the equation $\mathfrak{L}_g \hat u = 0$.  
This gives rise to a Dirichlet-to-Robin type operator defined on a dense subspace of $L^2(\partial M)$ which is formally defined as follows:

\begin{equation*}
\begin{aligned}
\mathrm{dom}(\mathfrak{B}_g)
= \Big\{&
u \in L^2(\partial M , g) \colon \,
\exists\, \hat u \in H^1(M , g)
\ \text{such that}\ 
\operatorname{Tr}\hat u = u, \
\mathfrak{L}_g \hat u = 0\\ & \ \text{in the weak sense, and }   \partial_\eta \hat u - \mathfrak qu \in L^2(\partial M, g)\Big\},
\end{aligned}
\end{equation*}
and, for $u\in \mathrm{dom}(\mathfrak B_g)$, we set
\[
\mathfrak B_g u \doteq \partial_\eta \hat u - \mathfrak q u.
\]

Since $\mathfrak q\in L^\infty(\partial M)$ and $u\in L^2(\partial M)$, the product 
$\mathfrak qu$ belongs to $L^2(\partial M)$.  
We say that the boundary expression 
\[
\partial_\eta \hat u - \mathfrak q u
\quad\text{belongs to } L^2(\partial M)
\]
if there exists $\psi\in L^2(\partial M)$ such that, for all $\varphi\in H^1(M)$,
\begin{equation*}\label{weak_boundary_identity_final}
\int_M \langle\nabla \hat u , \nabla \varphi\rangle_g \, dv_g
+ \int_M \mathfrak p \hat u\, \varphi \, dv_g
- \int_{\partial M} \mathfrak q u \, \operatorname{Tr}\varphi \, da_g
=
\int_{\partial M} \psi  \operatorname{Tr}\varphi \, da_g.
\end{equation*}

\medskip

A number $\sigma \in \mathbb{R}$ is called an eigenvalue of the Steklov-type problem \eqref{eq_auxi} if there exists 
$u \in \mathrm{dom}(\mathfrak{B}_g)$, $u \not\equiv 0$, such that
\[
\mathfrak{B}_g u = \sigma u.
\]

\medskip

Associated to the boundary operator $\mathcal{B}_g$, we consider the bilinear form
\[
\mathfrak{a} (u,v)
=
\int_M \Bigl\langle \nabla \hat u , \nabla \hat v \Bigr\rangle_g \, dv_g
+
\int_M \mathfrak p\, \hat u\, \hat v \, dv_g
-
\int_{\partial M} \mathfrak q\, u\, v \, da_g,
\]
defined for $u,v \in H^{1/2}(\partial M , g)$, where $\hat u,\hat v$ are the
corresponding weak solutions of $\mathfrak L_g \hat u = \mathfrak L_g \hat v = 0$
with boundary traces $u$ and $v$, respectively. The form
$\mathfrak{a}$ is symmetric and the associated quadratic form
\[
\mathfrak{a}[u] \doteq \mathfrak{a}(u,u)
\]
is closed on $H^{1/2}(\partial M , g)$. Since the trace operator is compact, it follows that $\mathfrak{B}_g$ has compact resolvent. Hence its spectrum is discrete. Moreover, adapting the arguments from \cite[Theorem 3.1]{AM00}, one can show that $\mathfrak{B}_g$ is also bounded from below.

In summary, if $0$ is not an eigenvalue of $\mathfrak{L}_g$,
the boundary operator $\mathfrak B_g$ is self-adjoint, bounded from below and its spectrum consists of a sequence of real eigenvalues of finite
multiplicity diverging to $+\infty$:
$$
\sigma_{1}(M,\mathfrak{B}_g) \leq \sigma_{2}(M,\mathfrak{B}_g) \leq \ldots \leq \sigma_{k}(M,\mathfrak{B}_g) \leq \cdots\nearrow\infty,
$$
where, as usual, each eigenvalue is repeated according to its multiplicity. The min-max principle then yields the variational characterization of the
eigenvalues 
\begin{equation*}
\sigma_{1}(M,\mathfrak{B}_g)=\inf \left\{\displaystyle \mathfrak{a}(u,u)/ \vert\vert u\vert\vert_{L^2(\partial M)}^{2} \  \colon u \in \text{dom}(\mathfrak{B}_g)\backslash \{0\}\right\},
\end{equation*}
and for $k\geq 2$
\begin{equation*}
\sigma_{k}(M,\mathfrak{B}_g)=\inf_{u \in \text{dom}(\mathfrak{B}_g)\backslash \{0\}} \left\{\mathfrak{a}(u,u)/ \vert\vert u\vert\vert_{L^2(\partial M)}^{2} \  \colon  \int_{\partial M}  u \phi_{j}da_g = 0\right\},
\end{equation*}
where $\phi_{j}$ are the eigenfunctions corresponding to the eigenvalue $\sigma_{j}(M,\mathfrak{B}_g)$, for $j=1, \ldots, k-1$.

\begin{remark}
When $0$ lies in the spectrum of $\mathfrak{L}_g$, the Dirichlet-to-Robin map and its spectrum can nevertheless be defined rigorously using the hidden compactness theory of Arendt, Elst, Kennedy, and Sauter \cite{AEKS} (see also \cite{AM00,AM,Cox,TRAN,TZ2}). We refer to \cite[Section~6]{Zhu} for a detailed exposition of this framework in the present setting.
\end{remark}

Given any real number $\sigma$,  the \emph{counting function} $\mathcal{N}_{\sigma}(\mathfrak{p},\mathfrak{q})$ is defined as the supremum of the dimensions of all vector spaces $V \subset C^{\infty}(M)$ such that
$$
\mathfrak{a}[\varphi] < \sigma \|\varphi\|_{L^{2}(M)}^2 \quad \text{for all } \varphi \in V, \, \varphi \not\equiv 0.
$$

Since $\mathfrak{B}_g$ with domain $C^{\infty}(M)$ admits an $L^{2}$-self-adjoint extension (see \cite{AM00} for details), we have
$$\mathcal{N}_{\sigma}(\mathfrak{p},\mathfrak{q}) = \dim \operatorname{Im} \mathbf{1}_{(-\infty, \sigma)}(\mathfrak{L}_g;
\mathfrak{B}_g).
$$
Hence, $\mathcal{N}_{\sigma}(\mathfrak{p},\mathfrak{q})$ represents the number of eigenvalues of \eqref{eq_auxi} that are smaller than $\sigma$, counted with multiplicity. In particular, for $\sigma = 0$, we use the notation
$$\textrm{Neg}(\mathfrak{p},\mathfrak{q})=\mathcal{N}_0(\mathfrak{p},\mathfrak{q}),$$
which denotes the number of negative eigenvalues of the Steklov-type problem.

For the purposes of this article, we would like to consider test functions for the Rayleigh quotient associated with $\mathfrak B_g$ in the whole space $H^1(M)$. However, as observed by Escobar in his addendum \cite{E-addendum}, the lower boundedness for the spectrum of the Dirichlet-to-Robin operator $\mathfrak{B}_g$ in $H^1(M)$ is completely
determined by the positivity of the first Dirichlet eigenvalue $\lambda_1(\mathfrak L_g^D)$ of the interior
operator $\mathfrak L_g$ (see also \cite[Remark 1]{LM}).

Based on Escobar’s observation, throughout this paper we shall assume that $\lambda_1(\mathfrak L_g^D)>0$. Under this hypothesis, we can take advantage of the following alternative characterization of the eigenvalues

\begin{equation*}
\sigma_{1}(M,\mathfrak{B}_g)=\inf \left\{\frac{\displaystyle \int_{
M}(|\nabla u|^{2}+\mathfrak pu^2) dv_g-\int_{\partial M}\mathfrak qu^2da_g}{\displaystyle\int_{\partial M}  u^{2} da_g} \colon u \in H^1(M), \text{Tr}\,u \neq 0\right\},
\end{equation*}
and for $k\geq 2$

\begin{equation*}
\sigma_{k}(M,\mathfrak{B}_g)=\inf \left\{\frac{\displaystyle \int_{
M}(|\nabla u|^{2}+\mathfrak pu^2) dv_g-\int_{\partial M}\mathfrak qu^2da_g}{\displaystyle\int_{\partial M}  u^{2} da_g}  \colon  \begin{array}{c} u \in H^1(M), \text{Tr}\,u \neq 0 \\[0.1cm] \text{and} \displaystyle \int_{\partial M}  u \phi_{j}da_g = 0 \end{array}\!\! \right\}\!,
\end{equation*}
where $\phi_{j}$ are the eigenfunctions corresponding to the eigenvalue $\sigma_{j}(M,\mathfrak{B}_g)$, with $j=1, \ldots, k-1$.  

This alternative characterization gives rise to the corresponding Courant-Fischer min-max principle 
\begin{equation*}
\sigma_k(M,\mathfrak{B}_g)
=
\min_{\substack{E \subset H^1(M) \\ \dim E = k}}
\;
\max_{\substack{u \in E \\ \operatorname{Tr} u \neq 0}}
\frac{\displaystyle \int_M \bigl(|\nabla u|^2 + \mathfrak{p} u^2\bigr)\, dv_g
- \int_{\partial M} \mathfrak{q} u^2\, da_g}
{\displaystyle \int_{\partial M} u^2\, da_g},
\end{equation*}
which constitutes the key tool in the strategy originating in \cite{K} for estimating high-order eigenvalues and will be implemented throughout this work.

\section{Technical aspects of the test functions}\label{sec_test-functions}

Throughout this paper, we will make use of special test functions adapted to the boundary problems under consideration. Their construction originates from the celebrated work of Korevaar \cite{K} and served as an inspiration for Kokarev \cite{kok}, who applied specific conformal diffeomorphisms of the sphere $\mathbb{S}^m$ to build suitable test functions.

A key point in the construction is the existence of a large collection of disjoint annuli carrying a controlled amount of mass. As shown by Grigor’yan and Yau \cite{gy}, this existence is guaranteed in a large class of metric spaces satisfying a global covering property, provided that the finite measure is non-atomic\footnote{$\mu \in \mathcal{M}^+(X)$ is non-atomic if, for any measurable set $A$ with $\mu(A)>0$, there exists a measurable subset $B$ of $A$ such that $\mu(A)>\mu(B)>0$.}.

In what follows, by an annulus in a manifold $N$ we mean any subset $A\subset N$ of the following form
$$A = \{x\in N\colon  r<d(x,a)<R\},$$
where $a \in N$ and $0 \leq r < R < +\infty$. We also denote by $2A$ the annulus
$$2A = \{x\in N \colon r/2<d(x,a)<2R\}.$$

The following proposition, due to Grigor’yan,  Netrusov, and Yau \cite[Theorem .~1]{gny}, lies at the core of the strategy for constructing disjoint annuli.

\begin{proposition}[Grigor’yan–Netrusov–Yau {\cite{gny}}]\label{non-atom}
Let \(N\) be a complete manifold satisfying the following covering property:
there exists a constant \(I > 0\) such that every ball of radius \(r\) in \(N\) can be covered by at most \(I\) balls of radius \(r/2\).
Let \(\mu\) be a non-atomic Radon measure on \(N\).
Then, for any positive integer \(k\), there exists a collection of annuli \(\{A_i\}_{i=1}^k\) in \(N\) such that the annuli \(2A_i\) are pairwise disjoint and, for each \(i = 1, \dots, k\),
\begin{equation*}
\mu(A_i) \ge c\, \frac{\mu(N)}{k},    
\end{equation*}
where \(c > 0\) is a constant depending only on \(I\).
\end{proposition}

For instance, the global $I$-covering property holds for all compact Riemannian manifolds and for all complete Riemannian manifolds with nonnegative Ricci curvature.

We now turn our attention to the construction of the test functions for the Rayleigh quotient associated with the Steklov-type problem \eqref{eq_auxi}, defined on an $n$-dimensional compact manifold $M$ with boundary $\partial M$ by
\begin{equation}\label{Rayleigh_quo_1}
    \mathcal{R}(u)=\frac{\displaystyle\int_{
M}(|\nabla u|^{2}+\mathfrak pu^2) dv_g-\int_{\partial M}\mathfrak qu^2da_g}{\displaystyle \int_{\partial M}  u^{2} da_g}.
\end{equation}
The test functions will be supported in pairwise disjoint annuli, chosen in correspondence with the family of disjoint annuli provided by Proposition~\ref{non-atom}.

Assume that $M$ admits a proper conformal map $\psi \colon M \rightarrow \mathbb{B}^m$, where $\mathbb{B}^m$ denotes the standard unit ball in $\mathbb{R}^m$.
Given a point $\theta\in \mathbb S^{m-1}$ and $t>0$, let $f_{\theta}(t) \colon \mathbb S^{m-1}\to \mathbb S^{m-1}$ be the conformal diffeomorphism given by
\begin{equation}
\label{def:f}
f_{\theta}(t) \doteq f^{-1}_\theta\circ s_t\circ f_\theta,
\end{equation}
where $f_\theta \colon \mathbb S^{{m-1}}\backslash\{\theta\}\to\mathbb R^{m-1}$ is the stereographic projection from the point $\theta$ onto the hyperplane $\{x\in\mathbb R^{m}\colon \langle x,\theta\rangle=0\}$, and $s_t \colon \mathbb R^{m-1}\to\mathbb R^{m-1}$ is the dilation $s_t(v)=tv$. For a fixed $R\in (0,\pi/2)$, one can choose a parameter $t=t(R)>0$ such that $f_{\theta}(t)$ maps the open ball $\tilde B_{2R}(\theta) \subset \mathbb{S}^{m-1}$ conformally onto the hemisphere $\mathbb{S}^{m-1}_+$ determined by $\theta$.

The first main step in our construction is to build a conformal extension of the function \(f_\theta(t)\) to the unit ball \(\mathbb{B}^m\).
Following \cite{E2}, we begin by considering the conformal map \(F \colon \mathbb{B}^m \to \mathbb{R}^m_+\) defined by
\begin{equation}\label{def_F}
F(x_1,\ldots,x_m) = \left(\frac{2x'}{(1+x_m)^2 + \lvert x'\rvert^2}, \frac{1-\lvert x\rvert^2}{(1+x_m)^2 + \lvert x'\rvert^2}\right),
\end{equation}
where \(x' = (x_1,\ldots,x_{m-1})\).
Its inverse map \(F^{-1} \colon \mathbb{R}^m_+ \to \mathbb{B}^m\) is given by
\[
F^{-1}(y_1,\ldots,y_{m-1},t) = \left(\frac{2y}{(1+t)^2 + \lvert y\rvert^2}, \frac{1 - t^2 - \lvert y\rvert^2}{(1+t)^2 + \lvert y\rvert^2}\right),
\]
where \(y = (y_1,\ldots,y_{m-1})\).
Observe that the restriction of \(F\) to the sphere is the stereographic projection from the south pole \(\theta_0 = (0,\ldots,0,-1) \in \mathbb{S}^{m-1}\), that is, \(F(\mathbb{S}^{m-1} \backslash \{\theta_0\}) = \partial \mathbb{R}^m_+\).
Moreover, \(F(\mathbb{D}^{m-1}) = \mathbb{S}^{m-1}_+ \subset \mathbb{R}^m_+\), where \(\mathbb{D}^{m-1} \subset \mathbb{B}^m\) denotes the equatorial disk.
For a given \(\theta \in \mathbb{S}^{m-1}\), composing \(F\) with a rotation we obtain an extension \(F_\theta\) of the stereographic projection \(f_\theta\) from the south pole \(\theta\).
We then define \(F_\theta(t) \colon \mathbb{B}^m \to \mathbb{B}^m\) by
\[
F_\theta(t) = F_\theta^{-1} \circ s_t \circ F_\theta,
\]
for \(t>0\) and \(\theta \in \mathbb{S}^{m-1}\), where \(s_t\) is the same dilation used in the definition of \(f_\theta(t)\).

For ease of notation, we write \(\mathcal{B}_{R_i} = B_{R_i} \cap \mathbb{B}^m\), and decompose its boundary as  
\[
\partial \mathcal{B}_{R_i} = \partial_0 \mathcal{B}_{R_i} \cup \partial_1 \mathcal{B}_{R_i},
\]
where
\[
\partial_0 \mathcal{B}_{R_i} = \partial \mathcal{B}_{R_i} \cap \operatorname{int}\mathbb{B}^m
\quad \text{and} \quad
\partial_1 \mathcal{B}_{R_i} = \partial \mathcal{B}_{R_i} \cap \mathbb{S}^{m-1}.
\]
For an annulus \(A \subset \mathbb{R}^m\), we also set \(\mathcal{A} = A \cap \mathbb{B}^m\).

For each annulus \(A_i = B_{R_i} \backslash B_{r_i} \subset \mathbb{R}^m\), let  
\(\tilde{A}_i = A_i \cap \mathbb{S}^{m-1}\) denote the annulus on the sphere centered at a point \(\theta_i \in \mathbb{S}^{m-1}\), with inner radius \(\tilde r_i \ge 0\) and outer radius \(\tilde R_i\), that is, \(\tilde{A}_i = \tilde B_{\tilde R_i} \backslash \tilde B_{\tilde r_i}\).  
For given \(\theta_i \in \mathbb{S}^{m-1}\) and \(\tilde R_i \in (0, \pi/2)\), we consider the conformal diffeomorphism \(f_{\theta_i}(t_i)\) such that  
\(f_{\theta_i}(t_i)(\tilde B_{2R_i}) = \mathbb{S}^{m-1}_+\),  
and denote its conformal extension to the unit ball \(\mathbb{B}^m\) by \(F_i \doteq F_{\theta_i}(t_i)\).

Let $\alpha_i \colon \partial F_i(\mathcal{B}_{2R_i}) \rightarrow \mathbb{R}$ be the $1$-Lipschitz function given by
$$\alpha_i(y)=\left\{
\begin{array}{rl}
\langle y,\theta_i \rangle & \text{ if~~} y\in \partial_1(F_i(\mathcal{B}_{2R_i})) = \mathbb{S}^{m-1}_+ ,\\[0.2cm]
0 & \text{ if~~} y \in \partial_0(F_i(\mathcal{B}_{2R_i})).
\end{array}
\right.
$$ 
We then take a Lipschitz extension of $\alpha_i$, that is, a function $\beta_i \colon F_i(\mathcal{B}_{2R_i}) \rightarrow \mathbb{R}$ such that 
$$ \text{Lip}_{F_i(\mathcal{B}_{2R_i})}(\beta_i) = \text{Lip}_{\partial F_i(\mathcal{B}_{2R_i})}(\alpha_i).$$
For instance, we can consider McShane's extension (see \cite{McShane}) defined by
$$\beta_i(x) = \sup_{y \in \partial F_i(\mathcal{B}_{2R_i})}\left\lbrace \alpha_i(y) - \vert x - y\vert\right\rbrace.$$
Finally, we extend $\beta_i$ to the whole ball $\mathbb{B}^m$ by setting $\beta_i(y) = 0$ for $y \notin F_i(\mathcal{B}_{2R_i})$. Note that $\beta_i$ is a $1$-Lipschitz function defined on $\mathbb{B}^m$ and supported in $F_i(\mathcal{B}_{2R_i})$.

As a last step, we define $\varPhi_i \colon \mathbb{B}^m \rightarrow \mathbb{R}$ by  $$\varPhi_i \doteq \beta_i \circ F_i .$$
Then the support of $\varPhi_i$ is contained in $\mathcal{B}_{2R_i}$. Moreover, the restriction of $\varPhi_i$ to $\mathbb{S}^{m-1}$ coincides with the function $\varphi_i \doteq \varphi_{R_i,\theta_i} \colon \mathbb{S}^{m-1} \rightarrow \mathbb{S}^{m-1}$ introduced by Kokarev \cite{kok}, defined by
\begin{equation}
\label{testf}
\varphi_i(x)=\left\{
\begin{array}{cc}
\langle f_{\theta_i}(t)(x),\theta_i \rangle & \text{ if~~} x\in \tilde B_{2R_i}(\theta_i),\\[0.1cm]
0 & \text{ if~~} x \notin \tilde B_{2R_i}(\theta_i),
\end{array}
\right.
\end{equation}
where $\tilde B_{2R_i}(\theta_i) \subset \mathbb{S}^{m-1}$ denotes the spherical ball of radius $2R_i$ centered at $\theta_i$.

Similarly, we define a function \(\bar \varPhi_i \colon \mathbb{B}^m \to \mathbb{R}\) by  
\[
\bar \varPhi_i \doteq \bar \beta_i \circ F_i,
\]  
where \(\bar \beta_i \colon \mathbb{B}^m \to \mathbb{R}\) is a 1-Lipschitz extension of the function
\[
\bar \beta_i(y) =
\begin{cases}
0 & \text{if } y \in \partial_0(F_i(\mathcal{B}_{\frac{r_i}{2}})), \\[0.1cm]
- \langle y, \theta_i \rangle & \text{if } y \in \partial_1(F_i(\mathcal{B}_{\frac{r_i}{2}})) = \mathbb{S}^{m-1}_+,
\end{cases}
\]  
and we set \(\bar \beta_i(y) = 0\) for \(y \notin F_i(\mathcal{B}_{\frac{r_i}{2}})\).  
Then \(\bar \varPhi_i\) is supported in the set \(\mathbb{B}^m \backslash \mathcal{B}_{\frac{r_i}{2}}\), and its restriction to \(\mathbb{S}^{m-1}\) coincides with the function \(\bar \varphi_i \doteq \bar \varphi_{r_i, \theta_i} \colon \mathbb{S}^{m-1} \to \mathbb{R}\) defined in \cite{kok} as
\begin{equation}\label{testf2}
\bar \varphi_i(x) =
\begin{cases}
0 & \text{if } x \in \tilde B_{\frac{r_i}{2}}(\theta_i), \\[0.1cm]
-\langle f_{\theta_i}(\tau)(x), \theta_i \rangle & \text{if } x \notin \tilde B_{\frac{r_i}{2}}(\theta_i),
\end{cases}
\end{equation}
where \(\tilde B_{\frac{r_i}{2}}(\theta_i) \subset \mathbb{S}^{m-1}\) is the spherical ball of radius \(\frac{r_i}{2}\) centered at \(\theta_i\).

For each $i = 1, \ldots, k+1$, we define the test function $u_i \colon M \rightarrow \mathbb{R}$ for the Rayleigh quotient as 
\begin{equation}\label{test_fuction_u}
u_i=\left\{
\begin{array}{rl}
(\varPhi_i\bar\varPhi_i) \circ \psi, &\text{ if~~ } \tilde r_i>0,\\[0.2cm]
\varPhi_i \circ \psi, & \text{ if~~ } \tilde r_i=0.
\end{array}
\right.
\end{equation}

\begin{lemma}\label{main_lemma_test}
For each $i=1,\ldots,k+1$, the function $u_{i}$ is Lipschitz and satisfies: 
    \begin{enumerate}
        \item [i)] $0\leq u_{i}(x)\leq1$ for every $x\in M$;
        \item [ii)] $u_{i}(x)\geq\dfrac{9}{25}$ for every $x\in \psi^{-1}(2\tilde{A}_{i}) = \psi^{-1}(2\mathcal{A}_{i}\cap \mathbb{S}^{m-1})$;
        \item [iii)] $\supp{u_{i}}\subset \psi^{-1}(2\mathcal{A}_{i})$.
\end{enumerate}
\end{lemma}
\begin{proof}
To prove $i)$, we first observe that 
\begin{eqnarray*}
    \alpha_{i}(y)=\langle y, \theta_{i}\rangle=|y||\theta_{i}|\cos{\theta}\in[-1,1],
\end{eqnarray*}
for every $y\in\mathbb{S}^{m-1}$ and $\theta_{i}\in\mathbb{S}_{+}^{m-1}$.
Hence, by definition of the MacShane's extension, we have
\begin{eqnarray*}
    \beta_{i}(z)=\sup_{}\{\alpha_{i}(y)-|z-y|\}\leq \alpha_{i}(y)\leq 1,
\end{eqnarray*}
for all $z \in \mathbb{B}^m$. Recalling that $\varPhi_i \doteq \beta_i \circ F_i $ and $\bar \varPhi_i \doteq \bar \beta_i \circ F_i $, we easily deduce the desired property for $u_i$.

Itens $ii)$ and $iii)$ follow directly by construction and the fact that the functions $\Phi$ and $\bar \Phi$ restricted to $\mathbb{S}^{m-1}$ coincide with the functions $\varphi_{r,\theta}$ and $\bar \varphi_{r,\theta}$ considered in \cite[Lemmata 2.3\&2.4]{kok}, respectively.
\end{proof}

Since the supports of the functions $u_i$ are contained in the subsets $\psi^{-1}(2\mathcal{A}_i)$, it follows that if the annuli $2\mathcal{A}_i$ are pairwise disjoint, then the sets $\psi^{-1}(2\mathcal{A}_i)$ are also pairwise disjoint. Consequently, the restriction of the functions $u_i$ to $\partial M$ are pairwise $L^2$-orthogonal.

\section{Spectral bounds and a Hersch type result}\label{sec_stek}

In this section, we investigate upper bounds for all eigenvalues of Steklov-type problems using conformal invariants. We establish Korevaar-type upper bounds for the spectra of the Dirichlet-to-Neumann and the conformal Dirichlet-to-Robin operators. For the latter operator, we also prove a Hersch-type result showing that the Euclidean metric uniquely maximizes the first normalized eigenvalue on the ball.

\subsection{Upper bounds for the Steklov spectrum}\label{upper_steklov}

 Assume that $(M,g)$ is an $n$-dimensional compact Riemannian manifold with boundary $\partial M$, admitting a conformal proper map $\psi \colon (M,g) \rightarrow P^m$. We introduce a conformal invariant that will play a key role in our eigenvalue estimates.
 
\begin{definition}
{\em The relative $m$-conformal volume of $\psi$}  is  defined  by
$$
      V_{rc}(M,m,\psi)=\sup_{f \in G} \vol(f (\psi (M))),
$$
where $G$ is the group of conformal diffeomorphisms of $P^m$. 
The {\em relative $m$-conformal volume of $M$} is then defined to be
$$
      V_{rc} (M,m)=\inf_{\psi} V_{rc} (M,m,\psi),
$$
where the infimum is taken over all non-degenerate conformal maps $\psi \colon M \rightarrow P^m$ such that  $\psi(\partial M) \subset  \partial P$.
\end{definition}

Assume that $(M,g)$ admits proper conformal immersions into the Euclidean unit ball $\mathbb{B}^m$. Recall that Fraser and Schoen \cite{fs} showed that the relative conformal volume provides the general upper bound \eqref{fraser_schoen} for the first nonzero Steklov eigenvalue.
We are interested in establishing upper bounds for the $k$-th Steklov eigenvalue. 
Since these eigenvalues are not invariant under scaling of the Riemannian metric, it is natural to consider the following normalized version
$$
\bar \sigma_k(M,g)\doteq\sigma_k(M,g)\vol_g(\partial M)^{\frac{1}{n-1}}.
$$
Our first main result provides a control of the normalized Steklov eigenvalue in terms of its relative $m$-conformal volume and the isoperimetric ratio. This result is inspired by the arguments in \cite{kok} and may be regarded as a higher-order generalization of \cite[Theorem 6.2]{fs}, as well as a variant of \cite[Theorem 1]{CEG}.

\begin{theorem}\label{thm:letterA}
 Let $(M^n,g)$ be a compact Riemannian manifold with non-empty boundary $\partial M$ and dimension $n\geqslant 2$, which admits a proper conformal immersion $\psi \colon M \to\mathbb{B}^m$. Then, for every $k\geqslant 1$, the $k$-th normalized Steklov eigenvalue of $M$ satisfies
\begin{equation}
\label{cvest}
\bar \sigma_k(M,g)\leq C \frac{V_{rc}(M,m)^{\frac{2}{n}}}{I(M)^{\frac{n-2}{n-1}}}k^{2/n},
\end{equation}
where $C=C(n,m)>0$ and $I(M)$ is the classical isoperimetric ratio, namely 
$$
I(M)=\frac{\vol_g(\partial M)}{\vol_g(M)^{\frac{n-1}{n}}
}.
$$
\end{theorem}

In particular, for Riemannian surfaces we have the following interesting consequence of Theorem \ref{thm:letterA}.

\begin{corollary} Let $M^2$ be a compact surface with non-empty boundary $\partial M$, which admits a proper conformal immersion $\psi \colon M \to \mathbb{B}^m$. Then
\begin{equation}\label{stekdim2}
         \sigma_k(M,g) L(\partial M) \leq C(m)  V_{rc}(M,m)k,
\end{equation}
    where $L(\partial M)$ is the length of $\partial M$.
\end{corollary}

\begin{remark}
According to the Weyl asymptotics for $\sigma_k(M,g)$, one would expect the right-hand side of \eqref{cvest} to scale like $k^{1/n}$ rather than $k^{2/n}$. This discrepancy may be explained by the fact that such an improved estimate would, in turn, yield an upper bound on the isoperimetric ratio $I(M)$, whenever $V_{rc}(M,m)<+\infty$. 
\end{remark}

\begin{remark}
Recently, Lima and Menezes \cite{LM} introduced Steklov eigenvalues with frequency for free boundary minimal surfaces in geodesic balls of $\mathbb{S}^m_{+}$, establishing upper bounds (see also the work of Medvedev \cite{Me} for the corresponding problem in $\mathbb{H}^m$, and \cite{MM} for related results in higher-dimensional submanifolds). A natural extension would be to obtain analogous bounds for conformal metrics on manifolds admitting proper immersions into such geodesic balls, using suitable conformal volume quantities.
\end{remark}

\begin{proof}[\bf Proof of Theorem \ref{thm:letterA}]
We consider the usual distance function $d_\text{can}$ on the Euclidean space, and recall that $(\mathbb{R}^m,d_\text{can})$ satisfies the global $I$-covering property. Following \cite{CEG}, we define a Borel measure $\mu$ supported on $\partial M$ as follows: for any open subset $O \subset \mathbb{R}^m$, we set
\begin{equation}\label{measure_thm4.1}
\mu(O) \doteq \int_{\psi^{-1}(O \cap \mathbb{S}^{m-1})} 1 \, da_{g}, 
\end{equation}
where $da_{g}$ denotes the volume measure on $\partial M$. In particular, we have $\mu(\mathbb{R}^m) = \vol_g(\partial M)$. 

Since the measure $\mu$ is non-atomic, we can apply Proposition \ref{non-atom} to the metric measure space $(\mathbb{R}^m,d_\text{can},\mu)$ to obtain a collection $\{A_i\}$ of $2(k+1)$ annuli in $\mathbb{R}^m$ such that 
\begin{equation}
\label{denom1}
\mu(A_i) \geq c \frac{\mu(\mathbb{R}^m)}{2(k+1)} \geq c \frac{\vol_g(\partial M)}{4k}, \qquad \text{for all } i=1,\ldots,2k+2,
\end{equation}
with $c>0$ depending only on $I$, and the collection $\{2A_i\}$ mutually disjoint. Reordering them if necessary, we can assume that the first $k+1$ annuli satisfy 
\begin{equation}
\label{denom2}
\vol_g(\psi^{-1}(2A_i\cap \mathbb{B}^m)) \leq \frac{\vol_g(M)}{k+1} \leq \frac{\vol_g(M)}{k}.
\end{equation}

Following \cite{kok}, it suffices to consider $k+1$ test functions for the Rayleigh quotient, each supported in disjoint domains of the unit ball, and then apply the characterization of $\sigma_k(M,g)$, namely 
\begin{equation} \label{eq:var}
     \sigma_k(M,g) \int_{\partial M} u_i^2 \; da_g \leq \int_M |\nabla u_i |^2 \; dv_g ,
\end{equation}
for $i = 1,\ldots,k+1$. 

We take the test function $u_i$ defined in \eqref{test_fuction_u}, whose support is contained in the subset $\psi^{-1}(2\mathcal{A}_i)$, where $\mathcal{A}_i = A_i \cap \mathbb{B}^m$. To estimate both sides of inequality \eqref{eq:var}, we will consider only the case $\tilde r_i > 0$, as the other case can be treated analogously. Applying H\"older's inequality, we obtain
\begin{eqnarray*}
\int_{M}\vert \nabla u_i\vert^2dv_g &\leqslant & 2\left(\int_{\supp u_i}\bar{\varPhi}_i^2 \vert\nabla \varPhi_i\vert^2 dv_g + \int_{\supp u_i} \varPhi_i^2 \vert\nabla \bar\varPhi_i\vert^2 dv_g\right)\\[0.2cm]
 &\leqslant & 2\left(\vert\vert\nabla \varPhi_i \vert\vert^2_{L^n(M)} + \vert\vert\nabla \bar\varPhi_i \vert\vert^2_{L^n(M)} \right)|\psi^{-1}(2\mathcal{A}_i)|^{\frac{n-2}{n}}.
\end{eqnarray*}

For the  first integral, since $\beta_i$ is $1$-Lipschitz, we have
\begin{eqnarray*}
\int_{M}\vert\nabla \varPhi_i \vert^n dv_g &\leqslant & \int_{M} \vert \nabla F_i\vert^n dv_g\\[0.2cm]
&=& n^{n/2}\vol_g(M,(F_i\circ\psi)^*g_\text{can})\\[0.1cm]
&\leqslant& n^{n/2}V_{rc}(M,m,\psi),
\end{eqnarray*}
where $g_\text{can}$ is the canonical round metric on $\mathbb{B}^m$. Similarly, we obtain
$$
\int_{M}\vert \nabla \bar\varPhi_i\vert^n dv_g\leqslant n^{n/2}V_{rc}(M,m,\psi).
$$
Combining the above inequalities and using \eqref{denom2}, we get the following estimate 
\begin{eqnarray}
\label{num}
\int_{M}\vert\nabla u_i\vert^2 dv_g &\leqslant& 4n V_{rc}(M,m,\psi)^{\frac{2}{n}} \left(\frac{\vol_g(M)}{k}\right)^{\frac{n-2}{n}}.
\end{eqnarray}
Now, using \eqref{denom1} we can also estimate the integral of $u_i$ on the boundary $\partial M$ as
\begin{equation}
\label{eq_den}
\int_{\partial M} u_i^2 da_g \geqslant \left(\frac{3}{5}\right)^4\mu(\mathcal{A}_i)\geqslant C  \frac{\vol_g(\partial M)}{k}.
\end{equation}
Substituting \eqref{num} and \eqref{eq_den} in \eqref{eq:var} we obtain 
\begin{equation}
\sigma_k(M,g)  \vol_g(\partial M)  \leqslant C \vol_g(M)^{\frac{n-2}{n}} V_{rc}(M,m)^{2/n} k^{2/n},
\end{equation}
where $C>0$ is a constant depending only on the dimensions $n$ and $m$. By reorganizing the terms to include the isoperimetric ratio, we then complete the proof. 
\end{proof}

\subsection{Spectrum of the conformal Dirichlet-to-Robin map}\label{escobar_sec}

In a remarkable work, Korevaar \cite{K} proved that for any positive integer k, the $k$-th eigenvalue $\lambda_k(\mathbb{S}^n,-\Delta_{\bar g})$ of the Laplacian satisfies
$$
\lambda_k(\mathbb{S}^n,-\Delta_{\bar g})\cdot \mbox{vol}_{\bar g}(\mathbb{S}^n)\leq C\cdot k^{\frac{2}{n}}, \qquad \text{for } \ \bar g \in [g],
$$
where $C(n)>0$ and $\bar g$ is conformal to the round metric of $\mathbb{S}^n,$ that is, $\bar g=u^{\frac{4}{n-2}}g$ for some positive smooth function $u.$ Note that $$\mbox{vol}_{\bar g}(\mathbb{S}^n)=\int_{\mathbb{S}^n}u^{\frac{2n}{n-2}}dv_g.$$

Inspired by Korevaar's work, Sire and Xu \cite{SX} introduced a new functional for the conformal spectrum of the conformal Laplacian on a closed manifold of dimension $n\geq 3$. In fact, they prove that 
 for any metric $\bar{g} \in [g]$, the $k$-th eigenvalue of the conformal Laplacian $L_{\tilde{g}},$ that will be defined appropriately later, satisfies the inequality
\[
\lambda_k(S^n, L_{\tilde{g}}) \int_{S^n} u^{\frac{4}{n-2}}  dv_g \leq C(n) k^{2/n},
\]
where $C(n)$ is a positive constant only depending on $n$.

For the Dirichlet-to-Neumann map, several results are available on upper bounds for higher eigenvalues (see, e.g., \cite{CEG,H}). Motivated by this line of research, we turn to the problem of establishing upper bounds for the spectrum of the conformal Dirichlet-to-Robin map on compact \(n\)-dimensional Riemannian manifolds \((M,g)\) with nonempty boundary \(\partial M\), \(n \geq 3\).

Given a metric $\bar g \in [g]$, with $\bar g = u^{\frac{4}{n-2}} g$, we recall that the the associated volume elements on $M$ and $\partial M$ transform as  
$
dv_{\bar g} = u^{\frac{2n}{n-2}} \, dv_g $ and $ 
da_{\bar g} = u^{\frac{2(n-1)}{n-2}} \, da_g,
$
respectively. We also recall the conformal Laplacian, the operator acting on $C^\infty(M)$ defined by
\[
L_g u \doteq -\Delta_g u + c_n R_g u,\quad\quad c_n = \frac{n-2}{4(n-1)},
\]
where $R_g$ denotes the scalar curvature of $(M,g)$ (some authors define $L_g$ to be $c_n^{-1}$ times our operator).  
The corresponding conformal boundary operator acting on $C^\infty(\partial M)$ is given by
\[
B_g u \doteq \frac{\partial u}{\partial \eta_g} + b_n H_g u,\quad\quad b_n = \frac{n-2}{2},
\]
where $H_g$ is the mean curvature of $\partial M$, and $\eta_g$ is the outward unit normal.

We are interested in obtaining upper bounds for the higher eigenvalues of the conformal Dirichlet-to-Robin problem
\begin{equation}\label{eing2}
\begin{cases}
L_g u = 0 & \text{ in }\  M, \\
B_gu= \sigma u & \text{ on } \ \partial M.
\end{cases} 
\end{equation}
Following the work of Sire and Xu \cite{SX}, we normalize the eigenvalues of the problem \eqref{eing2} as follows:
\begin{equation}\label{normalized_eig_conf}
\bar{\sigma}_k(M,B_{\bar g})=\sigma_k(M,B_{\bar g})\cdot\int_{\partial M}u^{\frac{2}{n-2}}da_g , \qquad \text{for } \ \bar g \in [g].
\end{equation}
It is not difficult to check that  this normalization is not invariant under a conformal change of the metric.

We begin by establishing a Korevaar-type bound for the spectrum of the conformal Dirichlet-to-Robin operator on the round unit ball  $(\mathbb{B}^n,g)$, inspired by \cite[Theorem 1.1]{SX}.

\begin{theorem}\label{thm:letterB}
Let $(\mathbb{B}^n, g)$ be the round unit ball. For every metric $\bar{g} \in [g]$, the normalized $k$-th eigenvalue of the conformal Dirichlet-to-Robin problem \eqref{eing2} on the ball $\mathbb{B}^n$ satisfies
\begin{eqnarray*}
\bar\sigma_k(\mathbb{B}^{n},B_{\bar g})\leqslant C(n)k^{\frac{2}{n}},
\end{eqnarray*}
for some constant $C(n)>0$.
\end{theorem}

\begin{remark}\label{remark_adaptation}
In light of Section \ref{sec_tec}, we note that the first Dirichlet eigenvalue of the conformal Laplacian $L_{\overline{g}}$ is positive, as is already the case for the round metric $g$. This follows from the conformal invariance of the sign of the eigenvalues, a fact that can be established by a straightforward adaptation of \cite[Proposition 3.1]{Say} to the Dirichlet setting.
\end{remark}

In what follows, we will make use of the following lemma, which generalizes \cite[Lemma 4.1]{SX}. Its proof is a straightforward adaptation and relies on the transformation laws for the scalar and mean curvatures under a conformal change:

\begin{equation}\label{confchange}
\begin{cases} 
R_{\bar g} = \frac{4(n-1)}{n-2}u^{-\frac{n+2}{n-2}} L_{g}u & \text { in } \ M, \\[0.1cm] 
 H_{\bar g}=\frac{2}{n-2}u^{-\frac{n}{n-2}}B_{g}u & \text { on } \ \partial M.\end{cases}
\end{equation}

\begin{lemma}\label{lem}
Let $(M,g)$ be a Riemannian manifold with boundary $\partial M$, and let $\bar{g} \in [g]$ be a conformal metric given by $\bar{g} = u^{4/(n-2)}g$ for some positive function $u \in C^\infty(M)$. Then, for any Lipschitz function $f \in \mathrm{Lip}(M) \cap C^\infty(M)$, the following identity holds:
\begin{multline*}
\int_{M}|\nabla_{\bar{g}}f|^{2}dv_{\bar g} + c_{n}\int_{M}R_{\bar{g}}f^{2}dv_{\bar g} + b_{n}\int_{\partial M}H_{\bar{g}}f^{2}da_{\bar g} =  \\[0.2cm]
\int_{M} |\nabla_{g}(u f)|^{2}dv_g + c_{n}\int_{M}R_{g}(u f)^2 dv_g + b_{n}\int_{\partial M}H_{g}(u f)^2da_g.
\end{multline*}
\end{lemma}

We now proceed to the proof of Theorem \ref{thm:letterB}.

\begin{proof}[\bf Proof of Theorem \ref{thm:letterB}]
As in the proof of Theorem \ref{thm:letterA}, we introduce a non-atomic Borel measure $\mu$ supported on $\mathbb{S}^{n-1}$, defined for every open subset $O \in \mathbb{R}^n$ as
$$
\mu(O) \doteq \int_{O\cap \mathbb{S}^{n-1}}u^{\frac{2}{n-2}}\,da_{g}.
$$
We then apply Proposition \ref{non-atom} to the metric measure space $(\mathbb{R}^n,d_{can},\mu)$ to find a collection $\{A_i\}$ of $2(k+1)$ annuli on $\mathbb{R}^n$ such that the annuli $\{2A_i\}$ are mutually disjoint, and
\begin{equation}
\label{denom11}
\mu(A_i)\geqslant c\frac{\mu(\mathbb{R}^n)}{4k}= \frac{c}{4k}\int_{\mathbb{S}^{n-1}} u^{\frac{2}{n-2}}da_g,
\end{equation}
for all $i=1,\ldots,2k+2$, with $c>0$ depending only on $n$. Moreover, the annuli can be reordered so that
\begin{equation}
\label{denom22}
\vol_g(2\mathcal{A}_i)\leqslant \frac{\vol_g(\mathbb{B}^{n})}{k} \ \ \text{and} \ \ \vol_g(2\tilde A_i)\leqslant \frac{\vol_g(\mathbb{S}^{n-1})}{k},
\end{equation}
where $2\mathcal{A}_i = 2A_i \cap \mathbb{B}^{n}$ and $2\tilde A_i = 2A_i\cap \mathbb{S}^{n-1}$.

Recall that the Rayleigh quotient for our problem is given by
\begin{equation}\label{rayleigh_kor}
    \mathcal{R}_{\bar g}(f) = \frac{\displaystyle \int_{\mathbb{B}^{n}} (|\nabla_{\bar{g}} f|^{2}  + c_{n}R_{\bar{g}}f^{2})dv_{\bar g} + b_{n} \int_{\mathbb{S}^{n-1}} H_{\bar{g}}f^{2}da_{\bar g}}{\displaystyle \int_{\mathbb{S}^{n-1}}f^{2}da_{\bar g}},
\end{equation}
where $b_{n}=\frac{n-2}{2}$ and $c_{n}=\frac{n-2}{4(n-1)}$. As test functions for the Rayleigh quotient, we consider $f_i = \frac{u_i}{u}$, where the functions $u_i$ are those constructed in Section \ref{sec_test-functions}, in this case with $\psi = \mathrm{Id} \colon \mathbb{B}^n \to \mathbb{B}^n$, and whose supports are contained in the mutually disjoint annuli $2\mathcal{A}_i$ for $i = 1, \ldots, k+1$.

Applying Lemma~\ref{main_lemma_test}, we obtain a lower bound for the denominator in~\eqref{rayleigh_kor} as follows:
\begin{eqnarray*}
\int_{\mathbb{S}^{n-1}} f_i^2 da_{\bar g} 
&=&  \int_{\mathbb{S}^{n-1}} u_i^2 u^{\frac{2}{n-2}}da_{g} \nonumber\\[0.2cm]
&\geqslant & \frac{3^4}{5^4}\int_{\mathbb{S}^{n-1}\cap 2{A}_i} u^{\frac{2}{n-2}}da_{g}.
\end{eqnarray*}
Therefore, from \eqref{denom11} there exists a constant $C>0$, depending only on the dimension $n$, such that 
\begin{equation}\label{esti_below}
\int_{\mathbb{S}^{n-1}} f_i^2 da_{\bar g} 
\geq \frac{3^4}{5^4}\mu(2{A}_i)\geq \frac{C}{k} \int_{\mathbb{S}^{n-1}} u^{\frac{2}{n-2}}da_g.
\end{equation}
For the numerator in \eqref{rayleigh_kor}, Lemma \ref{lem} yields
\begin{multline*}
\int_{\mathbb{B}^{n}}(|\nabla_{\bar{g}} f_i|^{2}dv_{\bar g} + c_{n}R_{\bar{g}}f_i^{2})dv_{\bar g} + b_{n}\int_{\mathbb{S}^{n-1}}H_{\bar{g}}f_i^{2}da_{\bar g} = \\[0.2cm]
 \int_{\mathbb{B}^{n}}|\nabla_{{g}} u_i|^{2}dv_{g} + b_{n}\int_{\mathbb{S}^{n-1}}H_{g}u_i^{2}da_{ g} \leq \\[0.2cm]
 \left( \int_{2\mathcal{A}_i} |\nabla u_i|^n \, dv_g \right)^{\frac{2}{n}}\vol_g(2\mathcal{A}_i)^{\frac{n-2}{n}} + b_{n}(n-1)\int_{2A_i\cap \mathbb{S}^{n-1}}u_i^{2}da_{g} .
\end{multline*}

We need to estimate the $L^n$-energy of the test functions $u_i$. By definition, case $\tilde r_i>0$, $u_i = \Phi_i \bar \Phi_i = (\beta_i \bar \beta_i)\circ F$, where $\beta_i$ and $\bar \beta_i$ are $1$-Lipschitz functions, and $F$ is a conformal diffeomorphism. By the conformal invariance of the energy and the fact that $\beta_i, \bar \beta_i \leq 1$, we deduce that
\begin{equation}\label{eq:conformal-invariance}
\int_{\mathbb{B}^n}|\nabla u_i|^n\,dv_{g}
= \int_{\mathbb{B}^n}|\nabla(\beta_i \bar \beta_i)|^n\,dv_g \leq 2^n \vol_g(\mathbb{B}^n).
\end{equation}

Thus, arguing as in the proof of Theorem \ref{thm:letterA}, and using  \eqref{denom2} and \eqref{eq:conformal-invariance}, we obtain
\begin{multline}\label{esti_abov}
\int_{\mathbb{B}^{n}}(|\nabla_{\bar{g}} f_i|^{2} + c_{n}R_{\bar{g}}f_i^{2})dv_{\bar g} + b_{n}\int_{\mathbb{S}^{n-1}}H_{\bar{g}}f_i^{2}da_{\bar g} \leq \\[0.2cm]
 (2^n\vol_g(\mathbb{B}^n))^{\frac{2}{n}} \frac{\vol_g(\mathbb{B}^n)^{\frac{n-2}{n}}}{k^{\frac{n-2}{n}}} + b_{n}(n-1)\frac{\vol_g(\mathbb{S}^{n-1})}{k} \leq\\[0.2cm]
 \left(4 \vol_g(\mathbb{B}^n) + b_{n}(n-1) \vol_g(\mathbb{S}^{n-1})\right) k^{-\frac{n-2}{n}}.
\end{multline}
Substituting \eqref{esti_below} and \eqref{esti_abov} into \eqref{rayleigh_kor} yields
\begin{eqnarray*}
    \mathcal{R}_{\bar g}(f_i)&\leqslant& \frac{C(n)}{\mu(\mathbb{R}^n)} k^{\frac{2}{n}}= C(n)\left( \int_{\mathbb{S}^{n-1}} u^{\frac{2}{n-2}}da_{g}\right)^{-1} k^{\frac{2}{n}},
\end{eqnarray*}
where $C(n) = 4 \vol_g(\mathbb{B}^n) + b_{n}(n-1) \vol_g(\mathbb{S}^{n-1})$. This finishes the proof of the theorem.
\end{proof}

We denote by $\mathcal{M}([g])$ the set of smooth metrics $\overline{g}$ in the conformal class $[g]$ such that the first Dirichlet eigenvalue of the conformal Laplacian $L_{\overline{g}}$, namely $\lambda_1^D(L_{\bar g})$, is positive (We refer to Section~\ref{sec_tec} for a justification of this restriction on the admissible metrics). Since the sign of $\lambda_k^D(L)$ is invariant under conformal change, it follows that $\mathcal{M}([g])$ is non-empty if and only if $\lambda_1^D(L)>0$, in which case $\mathcal{M}([g])$ coincides with the entire conformal class $[g]$, see the discussion in Remark \ref{remark_adaptation}.

Our second goal in this subsection is to improve the estimates obtained in Theorem \ref{thm:letterB} for the class of compact Riemannian manifolds $(M^n,g)$ with non-empty boundary $\partial M$ and dimension $n\geq 3$, which admit a proper conformal immersion $\psi \colon (M,g) \to\mathbb{B}^m$. The following result generalizes~\cite[Theorem 5.1]{SX}.

\begin{theorem}\label{thm:letterC}
If $(M,g)$ admits a proper conformal immersion $\psi \colon M \to\mathbb{B}^m$, then,
for every $\bar{g} \in \mathcal{M}([g])$, the normalized $k$-th eigenvalue of the conformal Dirichlet-to-Robin problem \eqref{eing2} on $M$   satisfies
   \begin{eqnarray*}
 \bar \sigma_k(M,B_{\bar g}) &\leq&\Bigg[4n V_{rc}(M,m,\psi)^{\frac{2}{n}} + c_n \vol_g(M)^{\frac{n-2}{n}} \vert R_g\vert_{L^{\frac{n}{2}}(M)} + \\
 & & b_n \vert H_g\vert_{L^{\frac{n}{2}}(\partial M)}  \vol_g(\partial M)^{\frac{n-2}{n}}\Bigg] k^{\frac{2}{n}}.
\end{eqnarray*}
\end{theorem}

\begin{remark}
 The requirement $\bar{g} \in \mathcal{M}([g])$ is a natural and non‑restrictive assumption in many geometric settings.
For instance, it is automatically satisfied if the scalar curvature is non‑negative in the metric $\bar g$.
\end{remark}

\begin{proof}[Proof of Theorem \ref{thm:letterC}]
The proof follows the same steps as the proof of Theorem \ref{thm:letterB}. Let $\psi \colon (M,g) \to\mathbb{B}^m$ be a proper conformal immersion. We consider the Borel measure
\[
\mu(O) \doteq \int_{\psi^{-1}(O \cap \mathbb{S}^{n-1})} u^{\frac{2}{n-2}} \, da_g,
\]
defined for all open subsets $O \subset \mathbb{R}^m$. As in the proof of Theorem~\ref{thm:letterB}, we apply Proposition~\ref{non-atom} to the metric measure space $(\mathbb{R}^m, d_{can}, \mu)$, obtaining a collection of $2(k+1)$ annuli $\{A_i\} \subset \mathbb{R}^m$ such that
\begin{equation}
\mu(A_i) \geqslant c \frac{\mu(\mathbb{R}^m)}{4k} = \frac{c}{4k} \int_{\partial M} u^{\frac{2}{n-2}} \, da_g,
\end{equation}
for every $i = 1, \ldots, 2k+2$, where $c > 0$ depends only on $n$. Since the annuli $\{2A_i\}$ are mutually disjoint, up to reordering the indices we also have
\begin{equation}
\label{denom3}
\vol_g(\psi^{-1}(2 \mathcal{A}_i)) \leqslant \frac{\vol_g(M)}{k} 
\quad \text{and} \quad 
\vol_g(\psi^{-1}(2\tilde A_i )) \leqslant \frac{\vol_g(\partial M)}{k}.
\end{equation}

For every $i = 1, \ldots, k+1$, we define $f_i \doteq \frac{u_i}{u}$, where $u_i$ are the test functions introduced in Section~\ref{sec_test-functions} and $u$ is the positive function satisfying $\bar g = u^{\frac{4}{n-2}} g$. Proceeding as in the proof of Theorem~\ref{thm:letterB}, we can estimate
\begin{equation}\label{numer_conf}
\int_{\partial M} f_i^2 \, da_{\bar g} \geq \frac{3^4}{5^4} \, \mu(2A_i) \geq  \frac{C}{k} \int_{\partial M} u^{\frac{2}{n-2}} \, da_g.
\end{equation}

On the other hand, from Lemma~\ref{lem} we have
\begin{multline}\label{ineq_f_i_1}
\int_M \left( |\nabla_{\bar g} f_i|^2 + c_n R_{\bar g} f_i^2 \right) dv_{\bar g} + b_n \int_{\partial M} H_{\bar g} f_i^2 \, da_{\bar g} \\[0.2cm]
= \int_M \left( |\nabla_g u_i|^2 + c_n R_g u_i^2 \right) dv_g + b_n \int_{\partial M} H_g u_i^2 \, da_{g} \\[0.2cm]
= \int_{\psi^{-1}(2\mathcal{A}_i)} \left( |\nabla_g u_i|^2 + c_n R_g u_i^2 \right) dv_g + b_n \int_{\psi^{-1}(2\tilde A_i)} H_g u_i^2 \, da_{g}.
\end{multline}

Following the argument in the proof of Theorem~\ref{thm:letterA} to estimate the gradient term (see \eqref{num}) and using \eqref{denom3}, we obtain
\begin{eqnarray}\label{ineq_numer1}
\int_{\psi^{-1}(2\mathcal{A}_i)} |\nabla_g u_i|^2 \, dv_g
&\leq & \left( \int_{2\mathcal{A}_i} |\nabla u_i|^n \, dv_g \right)^{\frac{2}{n}} \vol_g(2\mathcal{A}_i)^{\frac{n-2}{n}} \\[0.2cm]
&\leq & 4 n \, V_{rc}(M,m,\psi)^{\frac{2}{n}} \left( \frac{\vol_g(M)}{k} \right)^{\frac{n-2}{n}}. \nonumber
\end{eqnarray}

The remaining terms in \eqref{ineq_f_i_1} can be estimated as
\begin{multline}\label{ineq_numer2}
c_n \int_{\psi^{-1}(2\mathcal{A}_i)} R_g u_i^2 \, dv_g + b_n \int_{\psi^{-1}(2\tilde A_i)} H_g u_i^2 \, da_g \leq \\[0.2cm]
 c_n \vert R_g\vert_{L^{\frac{n}{2}}(\psi^{-1}(2\mathcal{A}_i))} \vert u_i\vert_{L^{\frac{2n}{n-2}}(\psi^{-1}(2\mathcal{A}_i))}^2 
+ b_n \vert H_g\vert_{L^{\frac{n}{2}}(\psi^{-1}(2\tilde{A}_i))} \vert u_i\vert_{L^{\frac{2n}{n-2}}(\psi^{-1}(2\tilde{A}_i))}^2  \leq\\[0.2cm]
 c_n \vert R_g\vert_{L^{\frac{n}{2}}(M)} \left( \frac{\vol_g(M)}{k} \right)^{\frac{n-2}{n}}
+ b_n \vert H_g\vert_{L^{\frac{n}{2}}(\partial M)} \left( \frac{\vol_g(\partial M)}{k} \right)^{\frac{n-2}{n}}.
\end{multline}
Combining \eqref{ineq_numer1} and \eqref{ineq_numer2}, we obtain
\begin{multline}
\int_{M}(|\nabla_{\bar{g}} f_i|^{2} + c_{n}R_{\bar{g}}f_i^{2})dv_{\bar g} + b_{n}\int_{\partial M}H_{\bar{g}}f_i^{2}da_{\bar g}\leq b_n \vert H_g\vert_{L^{\frac{n}{2}}(\partial M)} \left(\frac{\vol(\partial M)}{k}\right)^{\frac{n-2}{n}}\nonumber\\[0.2cm]
+  \left(\frac{\vol_g(M)}{k} \right)^{\frac{n-2}{n}}\left[4n V_{rc}(M,m,\psi)^{\frac{2}{n}} + c_n \vert R_g\vert_{L^{\frac{n}{2}}(M)}\right].    
\end{multline}\label{esti_abov1}
Substituting \eqref{numer_conf} and \eqref{esti_abov1} into the corresponding Rayleigh quotient, we conclude that
  \begin{eqnarray*}
  \sigma_k(M,B_{\bar{g}})\cdot\int_{\partial M}u^{\frac{2}{n-2}}da_g &\leq&\Bigg[4n V_{rc}(M,m,\psi)^{\frac{2}{n}} + c_n \vol_g(M)^{\frac{n-2}{n}} \vert R_g\vert_{L^{\frac{n}{2}}(M)} +\\[0.2cm]
 & & b_n \vert H_g\vert_{L^{\frac{n}{2}}(\partial M)} \vol_g(\partial M)^{\frac{n-2}{n}}\Bigg] k^{\frac{2}{n}}.
\end{eqnarray*}
\end{proof}

\subsection{Hersch-type result}\label{hersch_section}
 For closed manifolds, the first result concerning extremal geometries for the Laplace–Beltrami eigenvalues $\lambda_k$ was obtained by Hersch~\cite{H} for the two-sphere. In fact, he proved that the standard metric maximizes the first Laplacian eigenvalue $\lambda_1$ among all Riemannian metrics of the same area, and moreover, it is the unique maximizer up to isometry. 
 In this subsection, we present a Hersch-type result in our setting.

Let $M^{n}$ be a compact Riemannian manifold with boundary $\partial M$. We introduce the following eigenvalue functional 
\begin{equation*}
    \overline{\sigma}_{k}(M,[g],B_{\bar g})=\sigma_{k}(M,B_{\bar g}) \vol_{\bar g}(\partial M,1/2), 
\end{equation*}
where $\bar{g} \in [g]$ is a metric conformally related to $g$ via $\bar{g} = u^{\frac{4}{n-2}} g$ for some positive smooth function $u$, and
\[
\vol_{\bar g}(\partial M,p) \doteq \int_{\partial M} u^{\frac{4p}{n-2}} \, da_g.
\]

Next, we consider the supremum over the conformal class
\begin{equation*}
    \sigma^*_{k}(M,B_g)=\sup_{\bar{g}\in[g]}\overline{\sigma}_{k}(M,[g],B_{\bar g}).
\end{equation*}

The following result addresses the geometric problem of determining whether the supremum $\sigma^*_{k}(M,B_g)$ is attained by a Riemannian metric in a particular case. In the literature, such a metric is referred to as maximal.

\begin{theorem}\label{thm:letterD}
  Let $(\mathbb{B}^{n},g)$ be the Euclidean unit ball. Then
\begin{equation*}
\sigma^*_{1}(\mathbb{B}^{n},B_g) = \overline{\sigma}_{1}(\mathbb{B}^{n},[g],B_g).
\end{equation*}
Moreover, if for some metric $\bar g\in[g]$ we have 
\begin{equation*}
\sigma^*_{1}(\mathbb{B}^{n},B_g) = \overline{\sigma}_{1}(\mathbb{B}^{n},[g],B_{\bar g}),
\end{equation*}
then $\bar g$ coincides with the Euclidean metric up to multiplication by a positive constant.
\end{theorem}
\begin{proof}
Considering the function $1/u$, where $\bar{g} = u^{\frac{4}{n-2}} g$, as a test function for the Rayleigh quotient, we have
\begin{align*}
\sigma_{1}(\mathbb{B}^{n},B_{\bar g}) 
&\leq \frac{\displaystyle \int_{\mathbb{B}^{n}} \Big( |\nabla_{\bar g}(1/u)|^{2} + c_n R_{\bar g} (1/u)^{2} \Big) dv_{\bar g} + b_n \int_{\mathbb{S}^{n-1}} H_{\bar g} (1/u)^{2} da_{\bar g}}{\displaystyle \int_{\mathbb{S}^{n-1}} (1/u)^{2} da_{\bar g}} \\
&\leq \frac{\displaystyle \int_{\mathbb{S}^{n-1}} b_n H_g \, da_g}{\displaystyle \int_{\mathbb{S}^{n-1}} u^{\frac{2}{n-2}} da_g},
\end{align*}
where we used Lemma~\ref{lem}. Since $\sigma_1(\mathbb{B}^n,B_g) = b_n H_g$, it follows that
\begin{equation}\label{eq_hersch_1}
\sigma_{1}(\mathbb{B}^{n},B_{\bar g}) \leq \frac{\overline{\sigma}_{1}(\mathbb{B}^{n},[g],B_g)}{\displaystyle \int_{\mathbb{S}^{n-1}} u^{\frac{2}{n-2}} da_g}.
\end{equation}
Therefore,
\begin{equation*}
\sigma_{1}(\mathbb{B}^{n},B_{\bar g}) \, \vol_{\bar g}(\mathbb{S}^{n-1},1/2)
= \sigma_{1}(\mathbb{B}^{n},B_{\bar g}) \cdot \int_{\mathbb{S}^{n-1}} u^{\frac{2}{n-2}} da_g
\leq \overline{\sigma}_{1}(\mathbb{B}^{n},[g],B_g).
\end{equation*}

This completes the first part of the result. If equality holds, $1/u$ is the eigenfunction associated with $\sigma_1(\mathbb{B},B_g)$. By the transformation law \eqref{confchange}, we have 
$$
\sigma_1(\mathbb{B}^n,B_{\bar g})u^{\frac{2}{n-2}}=u^{\frac{n}{n-2}}\left(\frac{\partial}{\partial \bar{\nu}}+b_{n}H_{\bar{g}}\right){(1/u)} = B_g1 = \sigma_1(\mathbb{B}^n,B_{ g}),$$
since equality holds on \eqref{eq_hersch_1}. 
Thus, we must have  $u$ constant.
\end{proof}

\section{Negative eigenvalues and index estimates}\label{negative eigen}

This section is devoted to the study of the negative eigenvalues and index estimates of the Steklov-type problem
\begin{equation}\label{steklov_schr_prob}
\begin{cases}
\begin{array}{rl}
\mathfrak{L}_g = -\Delta_g u + \mathfrak{p}\,u = 0 & \text{in } M, \\[0.2cm]
\mathfrak{B}_g = \dfrac{\partial u}{\partial \eta_g} - \mathfrak{q}\,u = \sigma\,u & \text{on } \partial M,
\end{array}
\end{cases}
\end{equation}
where $\eta_g$ denotes the outward unit normal vector field along $\partial M$. Our main goal is to derive an estimate for the number of negative eigenvalues of problem \eqref{steklov_schr_prob}, in the spirit of the bound obtained in \cite[Theorem 1.2]{kok}, which in turn generalizes estimates from \cite{gny}. For our next result we will assume the following conditions on the potential functions 
\begin{equation}\label{cond_potential_func}
\mathfrak{p} \in L^{p}(M), \ \  p > \tfrac{n}{2}, 
\qquad \text{and} \qquad \mathfrak{q} \in L_+^\infty(\partial M), \ \  \mathfrak{q}\not\equiv 0.
\end{equation}

\begin{theorem}\label{thm:letterE}
Let $(M^{n},g)$ be a compact Riemannian manifold with $\partial M \neq \emptyset$ and $n\geq 2$. 
Let $m \in \mathbb{Z}_+$ such that the $m$-dimensional conformal volume of $M$ is defined. If $\lambda_1(\mathfrak L_g^D)>0$ and condition \eqref{cond_potential_func} is satified, then the number of negative eigenvalues of problem \eqref{steklov_schr_prob} satisfies
\begin{equation}\label{ineq_neg}
   \mathrm{Neg}(\mathfrak{p},\mathfrak{q}) 
   \geq 
   \frac{C\,\mathrm{vol}_g(M)^{\frac{2-n}{2}}}
   {\displaystyle
   \left[
   V_{rc}(M,m)^{\frac{2}{n}}
   + \vert {\mathfrak{p}}_{+}\vert_{L^\frac{n}{2}(M)}
   \right]^{\!\frac{n}{2}}
   }
   \left(
   \int_{\partial M} \mathfrak{q}\,da_g 
   \right)^{\!\frac{n}{2}},
\end{equation}
where $C>0$ is a constant depending only on $n$ and $m$.
\end{theorem}

\begin{proof}
Following the notation in the proof of Theorem \ref{thm:letterA}, let $\psi \colon M \rightarrow \mathbb{B}^m$ be a proper conformal immersion. In accordance with the proof of \cite[Theorem 1.12]{kok}, we consider 
$(\mathbb{B}^m, \omega)$, where $\omega$ denotes the push-forward of the 
Riemannian volume measure $\mathrm{vol}_g$ on $M$ under the immersion $\psi$. 
We also consider the metric measure space $(\mathbb{R}^m, d_{\mathrm{can}}, \mu)$, 
where $\mu$ is a non-atomic Borel measure defined for every open set 
$O \subset \mathbb{R}^m$ by 
\[
\mu(O) = \int_{\psi^{-1}(O \cap \mathbb{S}^{m-1})} \mathfrak{q}\, da_g.
\]
We can invoke Proposition \ref{non-atom} to guarantee the existence of $2(k+1)$ mutually disjoint annuli $\{A_i\} \subset \mathbb{R}^m$ satisfying
\begin{equation}\label{eq_mu_thm5}
\mu(A_i)\geq c\frac{\mu(\mathbb{R}^m)}{2(k+1)}\geq \frac{c}{4k}\int_{\partial M}\mathfrak{q}\, da_g ,
\end{equation}
for every $i=1,\ldots,2k+2$. Since the annuli $2A_i$ are mutually disjoint, we also have 
\begin{equation}
\mu(2\mathcal{A}_i)\leqslant \frac{\mu(\mathbb{B}^{m})}{k+1}\leqslant \frac{\vol_g(M)}{k}.
\end{equation}

We then set $k\geq 0$ to be the integer part of
\begin{equation*}
    \left(\frac{9c}{2500}\right)^{\frac{n}{2}}\frac{\vol_g(M)^{\frac{2-n}{2}}}{n^{\frac{n}{2}}\left[V_{rc}(M,m)^{\frac{2}{n}}+\displaystyle\vert {\mathfrak{p}}_{+}\vert_{L^\frac{n}{2}(M)}\right]^{\frac{n}{2}}}\left(\int_{\partial M}\mathfrak{q}da_g\right)^{\frac{n}{2}},
\end{equation*}
where $c$ is the constant appearing in \eqref{eq_mu_thm5}. The remainder of the proof consists in showing that there exist $k+1$ test functions $u_i$, with mutually disjoint supports, such that
    \begin{equation}\label{ineq_neg_count}
        \int_{M}|\nabla u_{i}|^{2}dv_g + \int_{M}\mathfrak{p}u^{2}_{i}dv_g< \int_{\partial M}\mathfrak{q} u^{2}_{i}da_g.
    \end{equation}
Indeed, this latter inequality implies that 
    \begin{equation*}
    \mathcal{N}(\mathfrak{p},\mathfrak{q}) \geqslant k+1\geq\frac{C(n,m)\vol_g(M)^{\frac{2-n}{2}}}{\left[V_c(m,M)^{\frac{2}{n}}+\displaystyle\vert {\mathfrak{p}}_{+}\vert_{L^\frac{n}{2}(M)}\right]^{\frac{n}{2}}} 
    \left( \int_{\partial M} \mathfrak{q}da_g \right)^{n/2}.
    \end{equation*}

Let $u_i$ be the Lipschitz test functions defined in~\eqref{test_fuction_u}, whose supports are contained in $\psi^{-1}(2\mathcal{A}_i)$. 
Using~\eqref{num} and H\"older inequality, the left-hand side of~\eqref{ineq_neg_count} can be estimated as
\begin{multline*}
\int_{M} \big(|\nabla u_{i}|^{2} + \mathfrak{p}u_{i}^{2}\big)\,dv_g 
\leqslant 4n\,V_{rc}(M,m,\psi)^{\frac{2}{n}}
\left(\frac{\mathrm{vol}_g(M)}{k}\right)^{\!\frac{n-2}{n}}
+ \int_{\psi^{-1}(2\mathcal{A}_i)} {\mathfrak{p}}_{+}\,u_i^2\,dv_g \\[0.2cm]
\leqslant  4n\,V_{rc}(M,m,\psi)^{\frac{2}{n}}
\left(\frac{\mathrm{vol}_g(M)}{k}\right)^{\!\frac{n-2}{n}}
+ \bigl|\mathfrak{p}_{+}\bigr|_{L^{\frac{n}{2}}(M)}^{2}
\left( \int_{\psi^{-1}(2\mathcal{A}_i)} u_i^{\frac{2n}{n-2}}\,dv_g \right)^{\!\frac{n-2}{n}} \\[0.2cm]
\leqslant  \left(\frac{\mathrm{vol}_g(M)}{k}\right)^{\!\frac{n-2}{n}}
\!\left[\, 4n\,V_{rc}(M,m,\psi)^{\frac{2}{n}} 
+ \bigl|\mathfrak{p}_{+}\bigr|_{L^{\frac{n}{2}}(M)}^{2}\right] \\[0.2cm]
<  \left(\frac{\mathrm{vol}_g(M)}{k}\right)^{\!\frac{n-2}{n}}
9n\!\left[ V_{rc}(M,m,\psi)^{\frac{2}{n}}  
+ \bigl|\mathfrak{p}_{+}\bigr|_{L^{\frac{n}{2}}(M)}^{2}\right].
\end{multline*}

On the other hand, the right-hand side of~\eqref{ineq_neg_count} satisfies
\begin{eqnarray*}
\int_{\partial M} \mathfrak{q}\,u_i^2\,da_g 
&\geq& \frac{81}{625}\int_{\psi^{-1}(A_{i}\cap\mathbb{S}^{m-1})} \mathfrak{q}\,da_g  \\[0.2cm]
&\geqslant& \frac{81c}{2500 k} \int_{\partial M}\mathfrak{q}\, da_g .
\end{eqnarray*}

Combining these two estimates, we obtain
\begin{eqnarray*}
\frac{\displaystyle \int_{M}\big(|\nabla u_{i}|^{2}+\mathfrak{p}u_{i}^{2}\big)\,dv_g}
{\displaystyle \int_{\partial M} \mathfrak{q}\,u_i^2\,da_g}
&<& 
\frac{2500n\,\mathrm{vol}_{g}(M)^{\frac{n-2}{n}}}{\displaystyle 9c\int_{\partial M}\mathfrak{q}\, da_g}
\!\left[V_{rc}(M,m)^{\frac{2}{n}}
+\bigl|\mathfrak{p}_{+}\bigr|_{L^{\frac{n}{2}}(M)}^{2}\right]k^{\frac{2}{n}}\\[0.2cm]
&\leqslant& 1,
\end{eqnarray*}
where the last inequality follows directly from the definition of~$k$. 
This concludes the proof.
\end{proof}

An interesting consequence of Theorem \ref{thm:letterE} can be derived for surfaces $(M^2,g)$ that admit a proper free boundary minimal immersion $\psi \colon M \rightarrow S_+^m$. In this case, after a composition with a conformal diffeomorphism $\Phi \colon \mathbb S^m_+ \to \mathbb B^m$, it is possible to obtain an upper bound for the relative $m$-conformal volume $V_{rc}(M,m, \Phi\circ \psi)$ in terms of the volume of $M$ (see \Cref{AppendixB}).

In this setting, we have the following

\begin{corollary}\label{cons_E}
Let $(M^{2},g)$ be a proper free boundary minimal surface immersed in $\mathbb{S}^m_+$, and assume that $\lambda_1(\mathfrak L_g^D)>0$ and condition \eqref{cond_potential_func} is satisfied. Then the number of negative eigenvalues of the problem \eqref{steklov_schr_prob} satisfies
\begin{equation}
   \mathrm{Neg}(\mathfrak{p},\mathfrak{q}) 
   \geq 
   \frac{C}
   {\displaystyle
   4\vol(M)
   + \vert {\mathfrak{p}}_{+}\vert_{L^1(M)}
   }
   \int_{\partial M} \mathfrak{q}\,da_g ,
\end{equation}
where $C>0$ is a constant depending only on $m$.
\end{corollary}

\subsection{Negative eigenvalues for the conformal Dirichlet-to-Robin map}

The conformal Dirichlet-to-Robin problem \eqref{eing2} plays a fundamental role in the Yamabe problem on manifolds with boundary. In a way analogous to the Euler characteristic for compact surfaces with boundary, its eigenvalues provide important geometric obstructions. For instance, if 
$\sigma_1(B_{g}) < 0$, then there exists a conformally related metric that is either scalar-flat with negative constant mean curvature, or has negative scalar curvature with minimal boundary. Moreover, a key feature of this spectral problem is that the sign of the first eigenvalue is conformally invariant (see \cite{E,E1}). The following result establishes that this invariance extends to all higher eigenvalues. A proof is given in \Cref{Appendix}.
\begin{proposition}\label{invariant}
 Assume that first Dirichlet eigenvalue of $L_{g}$ is positive. The sign of $\sigma_k(B_g)$ is invariant under conformal changes of the metric within the conformal class $[g]$.
\end{proposition}

This conformal invariance motivates the introduction of a fundamental invariant that captures the spectral structure of the conformal class. We define the \emph{negative eigenvalue counting function} of the conformal class $[g]$ as
$$
\operatorname{Neg}([g]) = \#\left\{k \in \mathbb{N} : \sigma_k(B_g) < 0\right\}.
$$
This integer-valued invariant provides a quantitative measure of the extent to which the conformal Dirichlet-to-Robin map fails to be positive definite within the given conformal class.

Since a counting function is analogous to the Morse index in variational problems, we offer insights into the stability properties of constant scalar curvature metrics with minimal boundary (or scalar flat with constant mean curvature on the boundary). Moreover, $\operatorname{Neg}([g])$ serves as an obstruction to certain conformal geometric properties, for instance, a conformal class admitting a metric with positive scalar curvature and minimal boundary must satisfy $\operatorname{Neg}([g]) = 0$.

We have the following consequence of Theorem~\ref{thm:letterE}. Although this result is essentially a corollary, we state it as a theorem due to its significance.

\begin{theorem}\label{conformal_neg}
Let $(M^{n},g)$ be a compact Riemannian manifold with boundary, of dimension $n\geq 3$. 
Assume that the mean curvature $H_g$ of $\partial M$ is a smooth nonpositive function. If $(M,g)$ admits a proper conformal map into the Euclidean unit ball $\mathbb{B}^{m}$, then, for every metric in $\mathcal{M}([g])$, the number of negative eigenvalues of the conformal Dirichlet-Robin map  satisfies
\begin{equation}\label{eq_comp}
    \operatorname{Neg}([g]) \geq C\frac{ \vol_{g}(M)^{\frac{2-n}{2}}}{\left[
   V_{rc}(M,m)^{\frac{2}{n}}
   + \vert (R_g)_{+}\vert_{L^\frac{n}{2}(M)}
   \right]^{\!\frac{n}{2}}} \left( \int_{\partial M} (-H_{g})da_{g} \right)^{n/2},
    \end{equation}
where \(C>0\) depends only on \(n\) and \(m\), and \(V_{\mathrm{rc}}(M,m)\) is the relative conformal volume of \(M\) of dimension \(m\).
\end{theorem}

\subsection{Morse index for type-II stationary hypersurfaces}\label{index_sec}
Let $(N^{n+1},g)$ be an oriented $(n+1)$-dimensional Riemannian manifold and $\Omega$ be a smooth compact domain in $N$. Let $x:\left(M^n, g\right) \rightarrow \Omega$ be a proper  isometric immersion of an orientable $n$-dimensional compact manifold $M$ with boundary $\partial M$ into $\Omega.$ The problem of finding  area-minimizing hypersurfaces among all such hypersurfaces in $\Omega$ which divides $\Omega$ into two disjoint domains $\Omega_1$ and $\Omega_2$ with prescribed wetting boundary area is called of Type-II partitioning problem (see Section \ref{sec_tec} for more details). Stationary hypersurfaces for the type-II partitioning problem in $\Omega$ are minimal hypersurfaces that intersect $\partial \Omega$ at a constant angle $\theta\in (0,\pi).$

For $M$ an immersed two-sided minimal hypersurface intersecting $\partial \Omega$ at a constant contact angle $\theta$, we consider the Jacobi operator  
$$J_{M}=\Delta_{M}+\text{Ric}^{\Omega}(\nu,\nu)+|A^{M}|^2,$$
where $\nu$ is the unit normal vector fields along $M$, $\text{Ric}^\Omega$ is the Ricci curvature tensor of $\Omega$  and $A^{M}$ is the second fundamental form of $M$. The associated boundary operator  
\[
B_{\partial M} = \frac{\partial}{\partial \eta} - \mathfrak{q}
\]
corresponds to an elliptic boundary condition for \(L_{M}\). Here,  
\[
\mathfrak{q} = \csc\theta A^{\partial \Omega}(\bar{\nu}, \bar{\nu}) + \cot \theta \, A^{M}(\eta, \eta),
\]
where \(\eta\) and \(\bar{\nu}\) denote the unit outward normal vectors to \(\partial M\) (with respect to \(M\)) and to \(\partial \Omega\), respectively. The tensor \(A^{\partial \Omega}\) is the second fundamental form of \(\partial \Omega\) in \(\Omega\).

The Type-II Morse index, denoted by $\operatorname{Ind}^{II}(M),$ is the number of negative eigenvalues of the problem
\begin{equation}\label{eigen_II}
\begin{cases}
\begin{array}{rl}
J_M \varphi = 0 & \text{in} \ M,  \\
B_{\partial M}\varphi = \sigma \varphi  &   \text{on} \ \partial M.
\end{array}
\end{cases}
\end{equation}
Equivalently, it is the dimension of the space of variations decreasing the area while preserving the wetting area. In this setting, we have the following application of Theorem \ref{thm:letterE}.

\begin{theorem}\label{morse_index}
Let $ M^n$ be a capillary two-sided  minimal hypersurface isometrically immersed in the unit ball $\mathbb{B}^{n+1}$. Assume that  the first Dirichlet eigenvalue of $-J_M$ is positive and $\csc\theta  + \cot \theta \,A^{M}(\eta, \eta)\geq 0$. Then
    \begin{equation}
   \operatorname{Ind}^{II}(M) \geq C\frac{\vol_g(M)^{\frac{2-n}{2}}}{V_{rc}(M,n+1)} \left( \int_{\partial M}[\csc\theta  + \cot \theta \,A^{M}(\eta, \eta)]\;da_g \right)^{n/2},
    \end{equation}
where $C=C(n)>0.$ \end{theorem}

\begin{remark}\label{remark_morseindex}
    It is worth noting that the proof of Theorem \ref{thm:letterE} remains valid for the following general  Jacobi-Robin type problem
\begin{equation}\label{steklov-schrodinger_problem3}
\begin{cases}
\begin{array}{rl}
\Delta_{M} u + \mathfrak{p}u = - \lambda u & \text{in} \ M,  \\[0.1cm]
\frac{\partial u}{\partial \eta}-\mathfrak{q}u = 0  &   \text{on} \ \partial M,
\end{array}
\end{cases}
\end{equation}
where $\mathfrak{p}$ and $\mathfrak{q}$ satisfy \eqref{cond_potential_func}. 
This formulation allows one to study the Morse index of compact capillary minimal hypersurfaces, see \Cref{AppendixC}.
\end{remark}

\appendix
\section{Conformal invariance of eigenvalue signs}\label{Appendix}

This appendix provides technical details regarding the conformal invariance of eigenvalue signs for the conformal Dirichlet-to-Robin operator.

Let us maintain the notations from Section~\ref{sec_tec}. It follows from Section \ref{sec_tec}  that the spectrum of Dirichlet-to-Robin map is discrete and can be listed 
in increasing order as
$$
\sigma_1(M,B_{\bar{g}}) < \sigma_2(M,B_{\bar{g}}) \leq \sigma_3(M,B_{\bar{g}}) \leq \cdots \leq \sigma_k(M,B_{\bar{g}}) \nearrow +\infty.
$$
We consider the case where $\bar{g}$ is conformal to $g$, i.e., $\bar{g} = u^{\frac{4}{n-2}}g$ for some smooth function $u>0$. Following the approach in \cite{Ho}, it is possible to give another min‑max characterization for the $k$-th eigenvalue $\sigma_k(M,B_{\bar{g}})$:
\begin{equation}\label{equation_minmax}
    \sigma_k(M,B_{\bar{g}}) = \inf_{\substack{E \subset \mathrm{dom}(B_{\bar{g}}) \\ \dim E = k}} \sup_{\substack{v \in E \\ \operatorname{Tr} v \neq 0}} \frac{G_{\bar g}(v)}{\displaystyle\int_{\partial M} v^2 \, u^{\frac{2}{n-2}} \, da_g},
\end{equation}
where 
$$
G_{\bar g}(v)
=
\int_M v\,L_g v\, dv_g
+
\int_{\partial M} v\,B_g v\, da_g.
$$
Since the proof follows the lines of the argument in \cite{Ho}, we only sketch the proof.

For $0<u \in C^{\infty}( M)$ and $v \in E$ with $u^{\frac{2}{n-2}} v \not \equiv 0$, define the functional
$$
F(u, v)\doteq \frac{G_{\bar g}(v)}{\displaystyle\int_{\partial M} u^{\frac{2}{n-2}} v^2 da_g}\left(\int_{\partial M} u^{\frac{2(n-1)}{n-2}} da_g\right)^{\frac{1}{n-1}}.
$$
For any nonzero $f \in C^{\infty}(\partial M)$, the conformal invariance of the operator $(L_g,B_g)$ implies that
$$
F^{\prime}(u, f)=\frac{G_{\bar g}(uf)}{\displaystyle\int_{\partial M} (uf)^2 u^{\frac{2}{n-2}}da_{g}}.
$$
Applying the min-max principle yields
$$
\sigma_k(M,B_{\bar{g}}) 
=\inf_{\substack{E \subset \mathrm{dom}(B_{\bar{g}}) \\ \dim E = k}} \sup_{\substack{v \in E \\ \operatorname{Tr} v \neq 0}}
      F'(u,f).
$$
Replacing $uf$ by $v$, we obtain \eqref{equation_minmax}.

The following result, inspired by \cite[Proposition 3.1]{Say}, establishes the conformal invariance of eigenvalue signs.

\begin{proposition}\label{sign}
 Assume that first Dirichlet eigenvalue of $L_{g}$ is positive. The sign of $\sigma_k(M,B_g)$ is invariant under conformal changes of the metric within the conformal class $[g]$.
\end{proposition}

\begin{proof}
 Let $\bar{g} = u^{\frac{4}{n-2}} g$ be a conformal metric with $u> 0$.  Without loss of generality,  assume that $\sigma_i(M,B_{\bar g})=0$ and $\sigma_i(M,B_g)>0$ (the remaining cases may be treated similarly). Then, the variational characterization \eqref{equation_minmax} yields
$$
\sigma_i(M,B_{\bar g})=\inf _{u_1, \ldots, u_i} \sup _{\sigma_1, \ldots, \sigma_i} \frac{G_g(w)}{\displaystyle \int_{\partial M} w^2\, u^{\frac{4}{n-2}}\, da_g},
$$
and
$$
\sigma_i(M,B_g)=\inf _{u_1, \ldots, u_i} \sup _{\sigma_1, \ldots, \sigma_i} \frac{G_g(w)}{\displaystyle \int_{\partial M} w^2\, da_g},
$$
where $w=\sum_{j=1}^i \sigma_j u_j.$

Let $v_1, \ldots, v_i$ be functions realizing  $\sigma_i(M,B_{\bar g}) = 0$, and set
$w=\sum_{j=1}^i \sigma_j v_j$. Then
$
\sup_{\sigma_1,\dots,\sigma_i} G_g(w)=0.
$
Therefore,
$$
\sigma_i(M,B_g)
\le
\sup_{\sigma_1,\dots,\sigma_i}
\frac{G_g(w)}{\displaystyle \int_{\partial M} w^2\, da_g}
=0,
$$
which contradicts the assumption that $\sigma_i(M,B_g)>0$.
\end{proof}

\section{Relative conformal volume and free boundary minimal surfaces.}\label{AppendixB}

It is well known that the relative $m$-conformal volume of a conformal map $\psi \colon M^n \to \mathbb{B}^m$ is bounded below by the volume of the $n$-dimensional unit ball $\mathbb{B}^n$ (see \cite[Remark 5.8]{fs}). In this appendix, we derive an upper bound for the relative conformal volume for surfaces that admit a free boundary minimal immersion into the hemisphere $\mathbb{S}_+^m$. 

We start establishing a comparison result between the relative $m$-conformal volume associated with the maps $\psi_B:M\to\mathbb{B}^m$ and $\psi_S:M\to\mathbb{S}_+^m$.

\begin{proposition}\label{prop_B1}
Let $M^n$ be a compact Riemannian manifold with boundary, and let
$\psi_B \colon M \to \mathbb{B}^m$ be a proper conformal immersion into the Euclidean unit ball
$\mathbb{B}^m$.  
Fix a conformal diffeomorphism $\Phi \colon \mathbb{S}^m_+ \to \mathbb{B}^m$, and define the conformal map
$
\psi_S \doteq \Phi^{-1} \circ \psi_B \colon M \to \mathbb{S}^m_+.
$
Then 
$$
C_1 \cdot V_{\mathrm{rc}}(M, m, \psi_S)\leq V_{\mathrm{rc}}(M, m, \psi_B) \le C_2 \cdot V_{\mathrm{rc}}(M, m, \psi_S),
$$ 
with $C_1$ and $C_2$ depending only on $m$, $n$, and $\Phi$.
\end{proposition}
\begin{proof}
Let $u \in C^\infty(\mathbb{S}^m_+)$ such that 
$\Phi^* g_{\mathbb{B}^m} = e^{2u} \, g_{\mathbb{S}^m_+}$ such that $a \le e^{2u} \le b$, for some positive constants $a \le b$. For any $n$-dimensional submanifold $A \subset \mathbb{S}^m_+$, we have \begin{equation}\label{volume}
a^{n/2} \, \mathrm{vol}_{g_{\mathbb{S}^m_+}}(A) \le \mathrm{vol}_{g_{\mathbb{B}^m}}(\Phi(A)) \le b^{n/2} \, \mathrm{vol}_{g_{\mathbb{S}^m_+}}(A).
\end{equation}

Let $G(N)$ denote the conformal group of a Riemannian manifold $N$. Given $\phi \in G(\mathbb{B}^m)$, we define 
$
\Psi \doteq \Phi^{-1} \circ \phi \circ \Phi \in G(\mathbb{S}^m_+).
$
Since $\psi_B = \Phi \circ \psi_S$, we have 
$$
\phi \circ \psi_B = \phi \circ \Phi \circ \psi_S = \Phi \circ \Psi \circ \psi_S,$$
and hence 
$\phi \circ \psi_B(M) = \Phi( \Psi \circ \psi_S(M) ).$
Applying inequality \eqref{volume}, we obtain
$$a^{n/2} \, \mathrm{vol}_{g_{\mathbb{S}^m_+}}(\Psi \circ \psi_S(M))\le\mathrm{vol}_{g_{\mathbb{B}^m}}(\phi \circ \psi_B(M)) \le b^{n/2} \, \mathrm{vol}_{g_{\mathbb{S}^m_+}}(\Psi \circ \psi_S(M)).
$$
Taking the supremum over all $\phi \in G(\mathbb{B}^m)$, and noting that $\phi \mapsto \Psi$ defines a bijection between $G(\mathbb{B}^m)$ and $G(\mathbb{S}^m_+)$, we conclude that 
$$
 a^{n/2} \cdot V_{\mathrm{rc}}(M, m, \psi_S)\leq V_{\mathrm{rc}}(M, m, \psi_B) \le b^{n/2} \cdot V_{\mathrm{rc}}(M, m, \psi_S).
$$ 
\end{proof}

\begin{remark}
For the standard stereographic projection from the south pole, the constants appearing in Proposition \ref{prop_B1} can be easily computed as $a =1$ and $b=2$. In particular, 
$$V_{\mathrm{rc}}(M, m, \psi_S) \leq V_{\mathrm{rc}}(M, m, \psi_B) \le 2^{\frac{n}{2}} \, V_{\mathrm{rc}}(M, m, \psi_S).$$
\end{remark}

Let $P^m$ be an orientable Riemannian manifold with boundary and let $M$ be a compact orientable surface with boundary properly immersed in $P$. We recall that 
\begin{equation*}\label{invar_for}
  \int_{M} (H^2-4K+4K_s) \;dv_g
\end{equation*}
is invariant under conformal changes of the metric on $P^m$, where $K_s$ is the sectional curvature of $P$  evaluated on the tangent plane to $M$. Indeed,  it is straightforward to check that $(k_1 - k_2)^2 dv_g$ is invariant under conformal changes of the metric on $P$, where $k_1$ and $k_2$ denote the principal curvatures of $M$. Integrating over $M$ and using Gauss equation and Gauss-Bonnet formula we prove the claim.

\begin{proposition} \label{prop_est_vrc_to_vol}
Let $M$ be a free boundary minimal surface isometrically immersed in $\mathbb S^m_+$ by a proper map $\psi \colon M \rightarrow \mathbb S^m_+$. Then
$$
V_{rc}(M,m,\psi)\leq \vol(M).
$$
 \end{proposition}
\begin{proof}
For any $f \in G(\mathbb{S}_+^m)$, we have 
$$
      \int_{M} (H^2-4K+4) \;dv
      =\int_{f(M )} (\tilde{H}^2-4\tilde{K}+4) \;d\tilde{v},
$$
where $d\tilde{v}$ denotes the induced area element on $f(M )$, and 
$\tilde{K}$ and $\tilde{H}$ denote the Gauss and mean curvatures of $f(M ).$ Then by the Gauss-Bonnet Theorem, and  since $\chi(M )=\chi(f(M))$, we obtain
\begin{align*} 
     \int_{M} H^2\;dv  +4\int _{\partial M } \kappa \;da +4\mbox{vol} (M)&=\; \int_{f(M )} \tilde{H}^2 \; d\tilde{v} 
+ 4 \int_{\partial f(M )} \tilde{\kappa} \;d\tilde{a} \notag +4\mbox{vol} (f(M))\\
& \geq \; 4\int_{\partial f(M )} \tilde{\kappa} \; d\tilde{a}+4\mbox{vol} ( f(M)),
\end{align*}
where $\kappa$ and $\tilde{\kappa}$ denote the geodesic curvature of $\partial M$ and $\partial f(M)$, respectively.  
Since $M$ and $f(M)$ are free boundary, we have $\tilde \kappa=\kappa=0$. 
Therefore, we conclude that
\begin{equation*}\label{volcomparison}
   V_{rc}(M ,m,\psi) \leq \vol(f(M)) \leq \vol(M). 
\end{equation*}
\end{proof}

\section{Index estimates for stationary capillary minimal hypersurfaces.}\label{AppendixC}

As mentioned in Remark \ref{remark_morseindex}, our results allow one to study the Morse index of compact capillary minimal hypersurfaces through the spectrum of their Jacobi operator with Robin boundary conditions,
\begin{equation}\label{eigen_I}
\begin{cases}
\begin{array}{rl}
J_M \varphi = -\lambda\varphi & \text{in} \ M,  \\
B_{\partial M}\varphi = 0  &   \text{on} \ \partial M,
\end{array}
\end{cases}
\end{equation}
where we adopt the notation introduced in Section \ref{index_sec}.

Index estimates for compact capillary hypersurfaces in Riemannian manifolds have been investigated in the works of Hong and Saturnino \cite{Hong}, and Longa \cite{Longa}, and references therein.    Related to the type-I partitioning problem, in the case of free boundary minimal surfaces (that is, when the contact angle is $\pi/2$), index bounds in terms of the surface topology have been recently obtained by Sargent \cite{Sa} and by Ambrozio, Carlotto, and Sharp \cite{ACS}, Aiex and Hong as well as by Cavalcante and de Oliveira \cite{CO0,CO}.

The above setting leads to lower bounds for the index of capillary minimal surfaces in the unit ball $\mathbb{B}^{n+1}$.

\begin{theorem}\label{morse_index_thm}
Let $ M^n$ be a capillary  two-sided minimal hypersurface properly  immersed in the unit ball $\mathbb{B}^{n+1}$. If $\csc\theta  + \cot \theta \,A^{M}(\eta, \eta)\geq 0$, then
$$\mbox{Ind}(M) \geq C\frac{\vol_g(M)^{\frac{2-n}{2}}}{V_{rc}(M,n+1)} \left( \int_{\partial M}\csc\theta  + \cot \theta \,A^{M}(\eta, \eta)\;da_g \right)^{n/2},
$$
where $C=C(n)>0.$
\end{theorem}

 
\end{document}